\documentclass[a4paper,11pt]{article}

\usepackage{hyperref}

\usepackage{soul}
\usepackage{makeidx}
\usepackage[T1]{fontenc}
\usepackage[dvips]{graphicx}
\usepackage{amsmath,amscd,amsthm}
\usepackage{amssymb}
\usepackage{tabularx}
\usepackage[dvips]{graphics}
\usepackage[latin1]{inputenc}
\graphicspath{c:/MISDOC~1}\makeindex 
\usepackage[dvipsnames,usenames]{color}

\newtheorem{prp}{Proposition}

\newtheorem{lmm}{Lemma}
\newtheorem{thr}{Theorem}
\newtheorem{dfn}{Definition}
\newtheorem{rmr}{Remark}
\newtheorem{crl}{Corollary}
\newtheorem{ejemplo}{Example}
\newtheorem{notation}{Notation}

\renewcommand{\min}{{\rm min}}
\newcommand{\N}{\mathbb{N}}
\newcommand{\Z}{\mathbb{Z}}

\newcommand{\R}{\mathbb{R}}
\newcommand{\C}{\mathbb{C}}

\newcommand{\fkg}{\mathfrak g}

\usepackage{color}

\newcommand{\loccit}{{\it loc.\,cit.}}
\newcommand{\HP}{{\rm HP}}

\newcommand{\supp}{{\rm Supp}}
\newcommand{\hp}{{\rm hp}}
\newcommand{\rank}{{\rm rank}}

\begin{document}
	
	\begin{center} 
		
		{\Large  \bf Bracket ideals and Hilbert polynomial of filiform Lie algebras}
		
	\end{center}
	
	\begin{center}
		{\bf F.J. Castro-Jim\'{e}nez and M. Ceballos}\footnote{ {\small FJCJ: Departamento de \'Algebra e IMUS. Facultad de Matem\'aticas, Universidad de Sevilla. C/ Tarfia s/n, 41012 Seville (Spain). {\bf castro@us.es} \\ MC: Dpto. de Ingenier\'{\i}a. Universidad Loyola Andaluc\'{\i}a, Av. de las Universidades 2, Dos Hermanas, 41704 Seville (Spain).  {\bf mceballos@uloyola.es}}} \end{center}

		\begin{abstract} 
			For a complex finite-dimensional filiform Lie algebra $\fkg$, we first study the bifiltration given by the bracket ideals $[C^k\fkg,C^\ell\fkg]$ and then the behavior of its associated bivariate Hilbert polynomial. This behavior depends in particular on two numerical invariants that measure, on one hand, certain properties of the centralizers in $\fkg$ of the ideals in the lower central sequence and, on the other hand, the dimension of the largest abelian ideal that appears in the lower central series. We give examples proving that the Hilbert polynomial can distinguish isomorphism classes of filiform Lie algebras that cannot be distinguished by the two aforementioned invariants.
		\end{abstract}
		
		\vspace{0.3cm}
		
		
		
		\section*{Introduction} \label{intro}
		
		For a complex Lie algebra $\fkg$, with Lie bracket $[\, , ]$, and an integer $k\geq 1$, we denote by $C^k\fkg$ the $k$-th ideal in the {\em lower central series} of $\fkg$, i.e. $C^1\fkg = \fkg$ and $C^k\fkg = [C^{k-1}\fkg,\fkg]$ for $k\geq 2.$ One has $[C^{k}\fkg,C^\ell\fkg] \subseteq C^{k+\ell}\fkg$ for $k,\ell \in \N$, $k,\ell \geq 1$ and the family of {\em bracket ideals} $[C^{k}\fkg,C^\ell\fkg]$ is called the {\em bracket bifiltration} of $\fkg$,  indexed by the semigroup $\Z_{\geq 1} \times \Z_{\geq 1}$; see Section \ref{hp_subsection_definition}. 
		
		A Lie algebra $\fkg$ is said to be {\em filiform} (\cite[III.1]{VEthesis}, \cite[Sect. 1.5]{VE}) if $\fkg$ has finite dimension $n\geq 2$ and $\dim(C^k\fkg)=n-k$ for $k=2,\ldots,n$. In particular, any filiform Lie algebra is nilpotent, and its nilpotency class equals its dimension. The set of filiform Lie algebras of a given dimension is a nonempty Zariski open subset of the algebraic set of nilpotent Lie algebras of the same dimension (\cite[III.1]{VEthesis}. If $\{e_1,\ldots,e_n\}$ is a basis of $\C^n$, the equalities $[e_i,e_h] = e_{h-1}$ for $3\leq h\leq n$ define a Lie algebra structure on $\C^n$, the remaining brackets being 0. This Lie algebra is filiform (\cite[III.1]{VEthesis}, \cite[4.4]{VE}) and it is  denoted $\fkg_0$ (or $\fkg_0^n$ if one wants to emphasize its dimension). It is often called the {\it model}  filiform Lie algebra of the given dimension. For $n=2,3,4$ the only filiform (up to isomorphism) Lie algebra is $\fkg_0$ (see [III.6 {\loccit}]).

		In \cite[III.4]{VEthesis} the author proves that if $\fkg$ is a filiform Lie algebra of dimension $n\geq 4$ then the bracket ideal $[C^2\fkg, C^{n-3}\fkg]$ is $\{0\}$ if $n$ is odd (i.e. the inclusion $[C^2\fkg, C^{n-3}\fkg] \subset C^{n-1}\fkg$ is strict since the last vector space has dimension 1). In {\loccit} it is also proved that when $n$ is even there are non-isomorphic, $n$-dimensional  filiform Lie algebras $\fkg,\fkg'$ satisfying $[C^2\fkg, C^{n-3}\fkg] = C^{n-1}\fkg$ and $[C^2\fkg', C^{n-3}\fkg'] = \{0\}.$ In this regard, the reader can see 
		Sections \ref{(4,5,8)} and \ref{(5,7,10)}.
		
		The proofs of previous results use the existence of a so-called {\it adapted basis} for any filiform Lie algebra.  
		A basis $\{e_1,\ldots,e_n\}$ of  a  filiform Lie algebra $\mathfrak{g}$ is called {\em adapted}, see \cite[III.2]{VEthesis}, \cite[Sec. 4.2]{VE}, if the following relations hold\footnote{To simplify notations, the definition we give here is slightly different from,  although  equivalent to,  the original one given in {\loccit}
		}
		\begin{equation} \label{eq2}\begin{array}{ll}
				\ [e_1, e_h] = e_{h-1} & \,\,\, {\rm for} \,\,\, 3 \leq h \leq n, \\
				\ [e_2, e_h] = 0 & \,\,\, {\rm for} \,\,\, 1 \leq h \leq n, \\
				\ [e_3,e_h] = 0 & \,\,\,{\rm for} \,\,\, 2 \leq h \leq n.
			\end{array}
		\end{equation}
		As proved in {\loccit} any filiform Lie algebra admits an adapted basis (the proof uses that the base field is infinite).

		In the present paper, using some results in {\loccit} and in \cite{CNT}, \cite{EGN}, and \cite{ENR1}, we prove several results on the bifiltration given by the bracket ideals $[C^k\fkg, C^{\ell}\fkg]$. 
		This bifiltration is encoded in the Hilbert polynomial $\HP_\fkg(t,s)$ of the given Lie algebra $\fkg$, see Subsection \ref{hp_subsection_definition}.
		
		We also use two numerical invariants, modulo isomorphism of Lie algebras, introduced in \cite{ENR1}. Both are defined for non--model filiform Lie algebras, and then we assume $\dim \fkg \geq 5$. 
		
		The first numerical invariant $z_1=z_1(\fkg)$ is defined as
		\begin{equation}\label{defz1} 
			z_1 = {\rm max} \{k \in {\mathbb N} \,| \,C_{\mathfrak{g}}
			(C^{n-k+2}\fkg) \supsetneq C^2\fkg \}
		\end{equation}
		where $C_{\mathfrak{g}}
		(\mathfrak{h})$ is the centralizer of a given Lie subalgebra $\mathfrak{h}$ of $\mathfrak{g}$, that is, the set of elements in $\fkg$ whose bracket with any element of $\mathfrak{h}$ is zero.

		The invariant $z_2=z_2(\fkg)$  is defined as  
		\begin{equation}\label{defz2}
			z_2 = \, {\rm max} \, \{ k
			\in {\mathbb N} \,|\, C^{n-k+1}\fkg\,\, {\rm is} \,\,
			{\rm abelian}\}.
		\end{equation} By definition, $C^{n-z_2+1}\fkg$
		is the largest abelian ideal in the lower central series of $\fkg$. These two invariants satisfy the following relations, see \cite[Th. 15]{ENR1}:
		\begin{equation}\label{Bound}
			4 \leq z_1 \leq z_2 < n \leq 2z_2-2.
		\end{equation}
		
		Given a $n$-dimensional non--model filiform Lie algebra $\fkg$, we use the triple $(z_1,z_2,n)$ to summarize the information about both invariants and the dimension of $\mathfrak g$. We say that $(z_1,z_2,n)$ is the triple associated with $\mathfrak g$. 
		
		Algebras with different triples are non--isomorphic. If one fixes an adapted basis $\{e_1,\ldots,e_n\}$ of $\fkg$ then one has 
		\begin{equation}\label{z_1-z_2-equivalent-def}
			z_1 = \, \min
			\{k \geq 4 \, | \, [e_k,e_n] \neq 0\},  \ \ \ \  z_2 = \, \min \{k \geq 4 \, | \, [e_k,e_{k+1}] \neq 0\}.
		\end{equation}
		Moreover, with respect to an adapted basis, the  ideals of the lower central series are given by  $C^1\, \fkg=\fkg$ and
		\begin{equation}\label{coroderivada} C^k \, \fkg = \langle e_2, \ldots, e_{n-k+1}  \rangle   \end{equation} where $2 \leq k \leq n-1$ and the angle brackets mean the  $\C$--vector space generated by the corresponding vectors. This is proved in \cite[III.2. Proposition]{VEthesis}.

		The content of this paper is as follows. In Section \ref{previousresults} we recall Theorem \ref{leygeneral} from \cite{CCN} and \cite{ENR2}, which describes the brackets of non--model filiform Lie algebras of dimension $n\geq 5$ with fixed invariants $(z_1,z_2)$. Previous theorem generalizes \cite[Lemme 1]{Bra}. 
		
		The parameters appearing in Theorem \ref{leygeneral} are subject to the polynomial relations deduced from the Jacobi identities and to certain open conditions imposed by the definition of invariants $z_1, z_2$; that is, the open conditions that come from inequalities $[e_{z_1},e_n] \neq 0$ and $[e_{z_2}, e_{z_2+1}]\neq 0$.
		
		Section \ref{main-results} contains the main results of this paper. We first collect, in Sections \ref{hp_subsection_definition} and \ref{theta-vector}, several properties of the bracket ideals $[C^k\fkg, C^\ell \fkg]$, the Hilbert polynomial $\HP_\fkg(t,s)$, and the $\theta$--vector associated with a filiform Lie algebra $\fkg$.
		
		In Section \ref{homogeneidad} we study a homogeneity property of the structure constants of a filiform Lie algebra associated with a given triple $(z_1,z_2,n)$. We include its proof for the sake of completeness. 
		This homogeneity result is then used in Corollaries \ref{ley(n-2,n-2,n)normalizada} and \ref{ley(z_1,n-2-n)normalizada}.
		
		Section \ref{Description} is devoted to the study of filiform Lie algebras $\fkg$ with invariant $z_2(\fkg)$ equal to $n-2$. Theorem \ref{isomorfia(n-2,n-2,n)} proves that there only exist two isomorphism classes of non-model filiform Lie algebras with triple $(n-2,n-2,n)$ for $n \geq 6$. When $z_1(\fkg) < z_2(\fkg)=n-2$, Proposition \ref{caso(z_1,n-2,n)} makes Theorem \ref{leygeneral} more precise proving, in particular, the vanishing of certain structure constants of the law of $\fkg$ or equivalently the nullity of certain bracket ideal.
		
		The computation of the Hilbert polynomial of a $n$--dimensional filiform Lie algebra $\fkg$ with $z_2(\fkg)=n-2$ is done in Section \ref{hilbert_z2=n-2}. Theorem \ref{Hp(z_1,n-2,n)} and Corollary \ref{n-z1-p-1-classes} state that the Hilbert polynomial can distinguish among a finite number of isomorphism classes of filiform Lie algebra with triple $(z_1,n-2,n)$ and $z_1<n-2$. The number of these classes tends to infinity when $n-z_1$ does. 
		
		In Section \ref{sporadicHP} we compute the Hilbert polynomials of certain families of filiform Lie algebras of dimension $n=8,9, 10$, with $z_2=n-3$ in all cases. We show that for $n=8,10$ (resp. $n=9$) the Hilbert polynomial distinguishes (resp. does not distinguish) isomorphism classes of the corresponding filiform Lie algebras.
		
		Certain computations have been performed using computer algebra systems {\tt Maple}, {\tt Singular} and  {\tt Macaulay2}, \cite{maple, Singular, Macaulay2}. For basic concepts and the usual notation on Lie algebras, we have followed, without specific reference to, the references \cite{GK,JAC,Serre,Var98}.

\section{Previous results}\label{previousresults} 

Unless stated otherwise, all considered filiform Lie algebras are non-model. In \cite{CCN},  following \cite[Theorems 3 and 5]{ENR2}, the authors gave the law of any filiform Lie algebra associated with the triple $(z_1,z_2,n)$ with respect to an adapted basis. For the convenience of the reader, we now recall here that result:

	\begin{thr}{\rm {\cite[Theorem 2]{CCN}}} \label{leygeneral}  Let $\mathfrak{g}$ be a $n$-dimensional filiform Lie algebra with associated triple $(z_1,z_2,n)$ and let $\{e_1, \ldots, e_n\}$ be an adapted basis of $\fkg$. Then, there exist complex numbers $\alpha_i$, $\gamma_j$ and $\beta_{k \ell}$, with $1 \leq i \leq z_2-z_1+1$, $1 \leq j \leq n-z_2-1$, $2 \leq \ell \leq n-z_{2}$, and $1 \leq k < z_{2}-z_{1}+\ell$, such that
		\begin{align*} [e_1,e_h]&= e_{h-1} \mbox{ \rm{ for }} 3 \leq h \leq n, \\
			[e_{z_1+i},e_{z_2+1}]&=\alpha_1 e_{i+2}+\alpha_2 e_{i+1}+\cdots+\alpha_{i+1} e_2 \,\, \mbox{ \rm{ for }} 0 \leq i\leq z_2-z_1, \\
			[e_{z_1},e_{z_2+j}]&=\alpha_1 e_{j+1}+\gamma_1\, e_j+ \cdots+\gamma_{j-1}\,e_2 \,\,  \mbox{ \rm{ for }} 2 \leq j \leq n-z_2,  \\
			[e_{z_{1}+k},e_{z_{2}+\ell}]&= \sum_{h=2}^{k+\ell} P_h\left([e_{z_{1}+k-1},e_{z_{2}+\ell}] +[e_{z_{1}+k},e_{z_{2}+\ell-1}]\right) e_{h+1}+ \beta_{k \ell} \, e_2, \\ &\phantom{= }{\mbox{\rm{\ for }}} 2 \leq \ell \leq n-z_{2},\,\,  1 \leq k < z_{2}-z_{1}+\ell.\\
		\end{align*}	
	\end{thr} 
	
	

	\begin{rmr}\label{map-Ph} 
		The $\C$-linear map $P_h : \mathfrak{g} \rightarrow \C$, for $h=1,\ldots,n$, associates to each vector $u\in \fkg$, its $h$-th coordinate with respect to the basis $\{e_h\}_{h=1}^n$, 
		that is,  $$u = \sum_{h=1}^{n} P_h(u) e_h.$$
	\end{rmr}

	\begin{rmr}\label{Brat-lemme-1} 
		Notice on one hand that if $z_2(\fkg)=n-1$ then the only parameters appearing in Theorem \ref{leygeneral} are the $\alpha_i$ and so, in this case, that theorem is equivalent to \cite[Lemme 1]{Bra}. 
		On the other hand, in Theorem \ref{leygeneral}, and by the definition of $z_1$ and $z_2$, the brackets $[e_{z_1},e_n]$ and $[e_{z_2},e_{z_2+1}]$ are both non zero; see equalities (\ref{z_1-z_2-equivalent-def}).
	\end{rmr} 
	
	\begin{notation}\label{notacion-vectorial-parametros}
		In the sequel, we use the following notation for the parameters appearing in Theorem \ref{leygeneral}: $\underline{\alpha} =(\alpha_1,\ldots, \alpha_{z_2-z_1+1})$, $\underline{\gamma} =(\gamma_1,\ldots,\gamma_{n-z_2-1})$  and $\underline{\beta}$ for the vector with coordinates $\beta_{k,\ell}$ for $2 \leq \ell \leq n-z_{2},\,\,  1 \leq k < z_{2}-z_{1}+\ell,$ lexicographically ordered. Notice that parameters $\underline{\gamma}$  and $\underline{\beta}$ only appear when $z_2\leq n-2.$ With the notations as in Theorem \ref{leygeneral}, we say that the parameters $(\underline{\alpha},\underline{\gamma}, \underline{\beta})$ are associated with $\fkg$ with respect to the adapted basis $\{e_1,\ldots,e_n\}$. 
	\end{notation}

	\section{Main results}\label{main-results}
	

	

	\subsection{Bracket bifiltration and Hilbert polynomial} \label{hp_subsection_definition} 
	
	Let $\mathfrak{S}$ be the semigroup $\Z_{\geq 1}\times\Z_{\geq 1}$ endowed with its usual partial ordering. If $\fkg$ is a Lie algebra, we denote by $F_{(\bullet,\bullet)}(\fkg) $, or simply $F_{(\bullet, \bullet)}$, the bifiltration on $\fkg$ indexed by $\mathfrak{S}$ defined by the bracket ideals $F_{(k,\ell)}=F_{(k,\ell)}(\fkg):= [C^k\fkg, C^\ell\fkg]$ where $(k,\ell)\in \mathfrak{S}.$ 
	This {\em bracket bifiltration} is decreasing in the sense that one has $F_{(k',\ell')} \subseteq F_{(k,\ell)}$ whenever $(k,\ell) \preceq (k',\ell')$ where $\preceq$ is the natural partial ordering in $\mathfrak{S}$.
	
	
	\begin{dfn} \label{hp-def-para-nilpotentes}
		If $\fkg$ is a nilpotent Lie algebra of finite dimension, the Hilbert polynomial of $\fkg$ is the bivariate polynomial associated with the bracket bifiltration of $\fkg$, that is, the symmetric  polynomial 
		\begin{equation}\label{hp-def}
			\begin{aligned}
				\HP_\fkg	&= \HP_\fkg(t,s):= \sum_{k\geq 1,\ell\geq 1} \dim\,[C^k\fkg,C^\ell\fkg]\,t^k \,s^\ell \in \Z[t,s].
			\end{aligned}	
		\end{equation}
	\end{dfn}
	
	From now on, $\fkg$ is a filiform Lie algebra of dimension $n\geq 2$. Then one has  
	\begin{equation}\label{hp-def}
		\begin{aligned}
			\HP_\fkg	
			&=(n-2)t s \,+   \sum_{2\leq k\leq n-2} (n-k-1)(t^k \,s + t\, s^k)\\  &\phantom{xx} + \,\sum_{k\geq 2,\ell\geq 2} \dim\,[C^k\fkg,C^\ell\fkg]\,t^k \,s^\ell.\\
		\end{aligned}	
	\end{equation}
	
	We also denote  $$\HP^{(0)}_\fkg= \HP^{(0)}_\fkg(t,s)= (n-2)t s \,+ \sum_{2\leq k\leq n-2} (n-k-1)(t^k \,s + t\,s^k).$$ 
	
	
	The degree of $\HP_\fkg$ is $n-1$.

	
	\begin{notation}\label{coeff-hp-and-k*}
		We denote $\hp_{\fkg,k,\ell}= \dim\,[C^k\fkg,C^\ell\fkg]$ the coefficient of the monomial $t^ks^\ell$ of the Hilbert polynomial $\HP_\fkg.$  We write $$\HP^{(2)}_\fkg = \HP_\fkg - \HP_\fkg^{(0)}.$$ The super-index $2$ indicates that the non-zero monomials in $\HP_\fkg^{(2)}$ have degree $\geq 2$.  If $r\in \R$, we denote $r^*=n+1-r$. 
	\end{notation}

The Hilbert polynomial of $\fkg$ has a special shape.   
\begin{lmm}{\rm [Arrow shape of $\HP$]} \label{arrowshape} With previous notations one has that 
	$\hp_{\fkg,k',\ell'} \leq \hp_{\fkg,k,\ell}$ whenever $(k,\ell) \preceq (k',\ell')$ where $\preceq$ is the natural partial ordering in $\mathfrak{S}$. Moreover, $$[C^{z_2^*} \fkg,C^{z_2^*} \fkg]=\{0\}; \quad [C^k\fkg,C^{k}\fkg]=[C^k\fkg,C^{k+1}\fkg], \quad {\rm for} \,\, k \geq 2$$ and then the coefficients of the three monomials $t^ks^k,\,t^{k+1}s^k,\, t^ks^{k+1}$ in $\HP_\fkg$ are identical, that is, $\hp_{\fkg,k,k}=\hp_{\fkg,k+1,k}=\hp_{\fkg,k,k+1}$ for $k\geq 2$. 
\end{lmm}


\begin{proof} The monotone property $\hp_{\fkg,k',\ell'} \leq \hp_{\fkg,k,\ell}$ follows from the fact that the bracket filtration is decreasing.  The first equality follows from the definition of $z_2$, see equality (\ref{defz2}). For the second one, it is enough to prove it for $k \leq n-2$. Fix an adapted basis of $\fkg$, as in Equation (\ref{eq2}). We have $C^k\fkg = \langle e_2,\ldots,e_{k^*}\rangle $ for $k=2,\ldots,n-1$ and the proof follows since $(k+1)^*=n-k=k^*-1.$
	
\end{proof}

\subsection{The $\theta$ vector and the support of $\HP$}\label{theta-vector} Let $\fkg$ a nilpotent Lie algebra of dimension $n\geq 1$. For $k=1,\ldots,n-1$, we denote by $\theta_k(\fkg)$, or simply $\theta_k$,  the minimum $\ell$ such that $[C^{k}\fkg,C^{\ell} \fkg] =  \{0\}$. 
Thus, if $\fkg$ is filiform of dimension $n$, one has $\theta_1 = n-1$ and $\theta_{n-1}=1$. 


We call $\theta(\fkg):=(\theta_1(\fkg),\ldots,\theta_{n-1}(\fkg))$ the $\theta$--vector associated with $\fkg$. Two nilpotent Lie algebras with different $\theta$--vector are not isomorphic.

\begin{rmr}\label{z2*contraz2*-1-and-C2-C(z1*+1)} From the definition of $z_2=z_2(\fkg) $, see equation (\ref{defz2}), we have $[C^{z_2^*}\fkg,C^{z_2^*}\fkg]=\{0\} $ and  $[C^{z_2^*-1}\fkg,C^{z_2^*-1}\fkg]\not=\{0\} $. Then by Lemma \ref{arrowshape}, one has $[C^{z_2^*-1}\fkg, C^{z_2^*}\fkg]\not=\{0\} $ and so $\theta_{z_2^*}= z_2^*$.

Note that if $n\geq 4$, from inequalities (\ref{Bound}) we get $$\frac{n+2}{2} \leq z_2 \leq n-1.$$ Last inequalities are equivalent to $$2 \leq z_2^*\leq \frac{n}{2}.$$

One also has $[C^2\fkg,C^{n-2}\fkg] \subseteq C^n\fkg =\{0\}$. Moreover, from the very definition of $z_1=z_1(\fkg)$, see equation (\ref{defz1}), one has $[C^2\fkg, C^{z_1^*+1}\fkg]=\{0\}$ and so $\theta_2 \leq z_1^*+1$. Notice also that since $4\leq z_1 \leq n-1$, then $3 \leq z_1^*+1 \leq n-2$. In particular, $z_1^*+1=n-2$ is equivalent to $z_1=4$. See Corollary \ref{theta2_r0}  for more details on the values of $\theta_2$. 

\end{rmr} 



We denote $\N^* = \N\setminus \{0\}$ and $$E_\fkg= \bigcup_{k=1}^{n-1} \left((k,\theta_k)+\N^2\right).$$ The set $E_\fkg$ is stable by $\N^2$--translations, that is, $E_\fkg + \N^2 = E_\fkg$. 

We define $E^*_\fkg := (\N^*)^2 \setminus E_\fkg$; which is a finite set. Since the bracket bifiltration, see Subsection \ref{hp_subsection_definition}, is a decreasing filtration, one has $$E^*_\fkg = \supp (\HP_\fkg):=\{(k,\ell)\in \N^2\,\vert \, \hp_{\fkg,k,\ell}\neq 0\},$$ and then $\HP_\fkg$ is a fully dense bivariate polynomial, meaning that if $(k,\ell)\prec (k',\ell')$ as elements of the partially ordered semigroup $\mathfrak{S}$, and if the coefficient $\hp_{\fkg,k',\ell'} \not=0$ then $\hp_{\fkg,k,\ell} \not=0$.

If $\fkg$ and $\fkg'$ are two filiform Lie algebras and $E^*_\fkg \not= E^*_{\fkg'}$ then $\fkg$ and $\fkg'$ are non-isomorphic. 
In Subsection \ref{(4,5,8)} we prove that the $\HP$--invariant is finer than the $E$--invariant, and then also finer than the invariant $\theta$--vector. 



\subsection{Homogeneity of the structure constants} \label{homogeneidad}

First, we begin giving the following homogeneity result, which will be used in the next subsections; see 
Corollary \ref{ley(z_1,n-2-n)normalizada}. It gives a basis change over the general law indicated in Theorem \ref{leygeneral}. Recall that in this theorem the parameters $(\underline{\alpha},\underline{\gamma}, \underline{\beta})$ 
are associated with $\fkg$ relative to one of its adapted bases; see notation \ref{notacion-vectorial-parametros}.  

	
	

	
	



\begin{prp}\label{cambiobase}
	Let $\fkg$ be a $n$-dimensional filiform Lie algebra associated with the triple $(z_1,z_2,n)$. If the parameters $(\underline{\alpha},\underline{\gamma}, \underline{\beta})$ are associated with $\fkg$ with respect to an adapted basis, then, for any $\lambda \in \C^{*}=\C\setminus \{0\}$, the parameters $\lambda (\underline{\alpha},\underline{\gamma}, \underline{\beta})$ are associated with $\fkg$ with respect to a possible different adapted basis. 
\end{prp}

\begin{proof}
	Let $\{e_i\}_{i=1}^n$ be an adapted basis of $\fkg$ as in Theorem \ref{leygeneral}.
	We consider $f_i=\lambda_i e_i$, with $\lambda_i \in \C^{*}$\, for $i=1, \ldots, n$.         
	Note that $\{f_i\}_{i=1}^n$ is an adapted basis of $\fkg$ if and only if $\lambda_1 \, \lambda_h=\lambda_{h-1}$, $\forall 3 \leq h \leq n$ (see Equation (\ref{eq2})). This is equivalent to the condition
	\begin{equation} \label{condicion-sobre-lambda_h}
		\lambda_h=\frac{\lambda_2}{\lambda_1^{h-2}}, \quad {\rm for}\,\, 3 \leq h \leq n.
	\end{equation}
	Equality (\ref{condicion-sobre-lambda_h}) holds, in particular for $\lambda_1=1$ and $\lambda_2=\lambda_h\in \C^*$ for $h=3,\ldots,n$ and we assume previous equalities in the sequel. Write $\lambda=\lambda_2$. We now prove that the parameters $\lambda(\underline{\alpha}, \underline{\gamma}, \underline{\beta})$ are associated with $\fkg$ with respect to the adapted basis $\{f_1,\ldots,f_n\}$.
	
	
	By definition of the invariant $z_2=z_2(\fkg)$, see equalities (\ref{z_1-z_2-equivalent-def}), one has  $[f_j,f_k]=0$, for $2 \leq j < k \leq z_2$ since $\langle f_2, \ldots, f_{z_2} \rangle=
	\langle \lambda\,e_2, \ldots, \lambda\,e_{z_2} \rangle = C^{z_2^*}\fkg$.  
	Notice that the map $P_h: \fkg \rightarrow \C$ depends on the fixed basis; see Remark \ref{map-Ph}. 
	concretely, 
	If $\overline{P}_h$ is the map associated with the basis $\{f_i\}_{i=1}^n$, then,  for  $h=1,\ldots,n$ and $u\in \fkg$, we have    
	$$\overline{P}_h(u)=\lambda_h \,P(u).$$
	Bearing that in mind and according to Theorem \ref{leygeneral},
	we have 
	
	\begin{align*}
		[f_{z_1+i},f_{z_2+1}]=&\, \lambda^2\,[e_{z_1+i},e_{z_2+1}]=\lambda^2\,(\alpha_1\,e_{i+2}+\cdots+\alpha_{i+1}\,e_2)= \lambda^2\,\left(\frac{\alpha_1}{\lambda}\,f_{i+2}+\cdots+\frac{\alpha_{i+1}}{\lambda}\,f_2\right)\\
		=& \,\lambda \,\left(\alpha_1\,f_{i+2}+\cdots+\alpha_{i+1}\,f_2\right), \,\, {\rm for} \,\, 0 \leq i \leq z_2-z_1. \\
		[f_{z_1},f_{z_2+j}]=&  \,\lambda^2\,[e_{z_1},e_{z_2+j}]=\lambda^2 \,(\alpha_1 e_{j+1}+\gamma_1\, e_j+ \cdots+\gamma_{j-1}\,e_2)\\
		=& \,\lambda^2\,\left( \frac{\alpha_{1}}{\lambda}\,f_{j+1}+\frac{\gamma_1}{\lambda}\,f_j+ \cdots+\frac{\gamma_{j-1}}{\lambda}f_2  \right)=\lambda\,\left( \alpha_{1}\,f_{j+1}+\gamma_1\,f_j+ \cdots+\gamma_{j-1}\,f_2  \right), \\
		&{\rm for} \,\,\,\, 2 \leq j \leq n-z_2. \\
		[f_{z_1+k},f_{z_2+\ell}]=& \,\lambda^2\,[e_{z_1+k},e_{z_2+\ell}]\\ =& \,\lambda^2 \,\left(\sum_{h=2}^{k+\ell} P_h\left([e_{z_{1}+k-1},e_{z_2+\ell}] +[e_{z_{1}+k},e_{z_2+\ell-1}]\right) e_{h+1} \right)+\lambda^2\, \beta_{k, \ell} \, e_2 \\
		=& \,\lambda^2
		\sum_{h=2}^{k+\ell} P_h\Big( \frac{1}{\lambda^2} [f_{z_1+k-1},f_{z_2+\ell}]+\frac{1}{\lambda^2}[f_{z_1+k},f_{z_2+\ell-1}] \Big)\frac{1}{\lambda} f_{h+1}  + \lambda^2
		\frac{\beta_{k,\ell}}{\lambda} f_2 \\
		=& \,			
		\sum_{h=2}^{k+\ell} \overline{P}_h\left([f_{z_{1}+k-1},f_{z_{2}+\ell}] +[f_{z_{1}+k},f_{z_{2}+\ell-1}]\right) f_{h+1} +\lambda\,\beta_{k,\ell}f_2, \\
		& \,\, {\rm for} \,\, 2 \leq \ell \leq n-z_{2},\,\,  1 \leq k < z_{2}-z_{1}+\ell.
	\end{align*}
\end{proof}

	

\begin{rmr}\label{remarkcambiobase}
	Under the assumptions of Proposition \ref{cambiobase} or Theorem \ref{leygeneral}, $[e_{z_1},e_n]$ and $[e_{z_2},e_{z_2+1}]$ cannot be zero. Therefore, we have two possible cases. The first case occurs when $\alpha_1 \neq 0$. The second case arises when $\alpha_1=0$ and then there exist $m, q$ such that $1 < m \leq z_2-z_1$, $1 \leq q \leq n-z_2-1$, $\alpha_{m} \neq 0$, and $\gamma_{q} \neq 0$. 
	
	First, we assume that $\alpha_1 \neq 0$. Then, if we consider $\lambda=\frac{1}{\alpha_1}$ in the proof of Proposition \ref{cambiobase}, we obtain the law 
	\begin{align*} [f_1,f_h]&= f_{h-1} \mbox{ \rm{ for }} 3 \leq h \leq n, \\
		[f_{z_1+i},f_{z_2+1}]&= f_{i+2}+\frac{\alpha_2}{\alpha_1} f_{i+1}+\cdots+\frac{\alpha_{i+1}}{\alpha_1} e_2 \,\, \mbox{ \rm{ for }} 0 \leq i\leq z_2-z_1, \\
		[f_{z_1},f_{z_2+j}]&= f_{j+1}+\frac{\gamma_1}{\alpha_1}\, f_j+ \cdots+\frac{\gamma_{j-1}}{\alpha_1}\,e_2 \,\,  \mbox{ \rm{ for }} 2 \leq j \leq n-z_2,  \\
		[f_{z_{1}+k},f_{z_{2}+\ell}]&= \sum_{h=2}^{k+\ell} P_h\left([f_{z_{1}+k-1},f_{z_{2}+\ell}] +[f_{z_{1}+k},f_{z_{2}+\ell-1}]\right) \,f_{h+1}+ \frac{\beta_{k \ell}}{\alpha_1} \, f_2, \\ &\phantom{= }{\mbox{\rm{\ for }}} 2 \leq \ell \leq n-z_{2},\,\,  1 \leq k < z_{2}-z_{1}+\ell.\\
	\end{align*}
	In case that $\alpha_1=0$, we choose $\lambda=\frac{1}{\gamma_q}$ in the proof of Proposition \ref{cambiobase}. Consequently, the following brackets are obtained: 
	\begin{align*} [f_1,f_h]&= f_{h-1} \mbox{ \rm{ for }} 3 \leq h \leq n, \\
		[f_{z_1+i},f_{z_2+1}]&=\frac{\alpha_2}{\gamma_q} f_{i+1}+\cdots+
		\frac{\alpha_{i+1}}{\gamma_q} f_{2}\,\, \mbox{ \rm{ for }} 0 \leq i\leq z_2-z_1, \\
		[f_{z_1},f_{z_2+j}]&=\frac{\gamma_1}{\gamma_q}\, f_j+  \cdots+\frac{\gamma_{q-1}}{\gamma_q}\,f_{j-q+2}+f_{j-q+1}+
		\frac{\gamma_{q+1}}{\gamma_q}\,f_{j-q}+
		\cdots+ \frac{\gamma_{j-1}}{\gamma_q} \,e_2 \\
		&  \mbox{ \rm{ for }} 2 \leq j \leq n-z_2,  \\
		[f_{z_{1}+k},f_{z_{2}+\ell}]&= \sum_{h=2}^{k+\ell} P_h\left([f_{z_{1}+k-1},f_{z_{2}+\ell}] +[f_{z_{1}+k},f_{z_{2}+\ell-1}]\right) \,f_{h+1}+ \frac{\beta_{k \ell}}{\gamma_q} \, f_2, \\ &\phantom{= }{\mbox{\rm{\ for }}} 2 \leq \ell \leq n-z_{2},\,\,  1 \leq k < z_{2}-z_{1}+\ell.\\
	\end{align*}		
	
\end{rmr}








\subsection{Description of filiform Lie algebras $\fkg$ with $z_2(\fkg)=n-2$}
\label{Description}

In this Section, we describe filiform Lie algebras $\fkg$ with $z_2(\fkg)=n-2$ and $n=\dim \fkg \geq 6$. 
Our results in this section generalize the case $z_2=n-1$ studied by Bratzlavsky \cite{Bra}.

\begin{rmr}\label{z_2yz_1estrella} 
By definition of $z_2$, 
and since $z_2=n-2$, one has $z_2^*=3$ and $C^3\fkg$ is abelian, but $C^{2}\fkg$ is not (see Equality (\ref{defz2})). 

\end{rmr}

\subsubsection{Filiform Lie algebras associated with $(n-2,n-2,n)$} \label{(n-2,n-2,n)}

We begin studying the filiform Lie algebras $\fkg$ with $z_1(\fkg)=z_2(\fkg)=n-2$. 
%


\begin{prp}\label{caso(n-2,n-2,n)new}
Assume 
$\fkg$ is a filiform Lie algebra associated with the triple $(n-2,n-2,n)$.  Then there exists an adapted basis $\{e_i\}_{i=1}^n$ of  $\fkg$ and complex parameters $\alpha,\gamma,\beta$ such that 
\begin{align*} \begin{split} \label{law(n-2,n-2,n)} [e_1,e_h]&= e_{h-1} \mbox{ \rm{ for }} 3 \leq h \leq n, \\ [e_{n-2},e_{n-1}]&=\alpha\, e_2, \\ [e_{n-2},e_n]&=\alpha\,e_3+\gamma\,e_2, \\ 
		[e_{n-1},e_n]&=\alpha\,e_4+\gamma\,e_3+\beta\, e_2. \end{split}
\end{align*} There are no closed restrictions for the parameters. The unique open restriction comes from the fact that $[C^2 \fkg, C^3 \fkg] \neq \{0\}$, that is, $\alpha \neq 0$. 
\end{prp}

%
\begin{proof} The law of $\fkg$ is deduced from Theorem \ref{leygeneral} writing $\alpha=\alpha_1$, $\gamma=\gamma_1$ and $\beta=\beta_{1,2}$. Recall that we assume $n=\dim \fkg \geq 6$.
The unique Jacobi identity that we must impose is $J(e_{n-2},e_{n-1},e_n)=0$, which is given by
$$J(e_{n-2},e_{n-1},e_n)=[[e_{n-2},e_{n-1}],e_n]+[[e_{n-1},e_{n}],e_{n-2}]+[[e_{n},e_{n-2}],e_{n-1}]=
$$ $$[\alpha \, e_2,e_n]+
[\alpha\,e_4 + \gamma\,e_3+\beta \,e_2,e_{n-2}]-[\alpha \,e_3+\gamma \, e_2,e_{n-1}]=0.$$ 
Consequently, there are no closed restrictions for the coefficients $\{\alpha, \gamma, \beta\}$. We know that $C^3 \fkg$ is abelian and $C^2 \fkg$ is not. According to Lemma \ref{arrowshape} and Remark \ref{z_1-z_2-equivalent-def}, \,$\{0\} \neq [C^2 \fkg, C^2 \fkg]=[C^2 \fkg, C^3 \fkg]$.  Moreover,
$$[C^2 \fkg, C^3 \fkg]=[\langle e_2, \ldots, e_{n-1} \rangle, \langle e_2, \ldots, e_{n-2}  \rangle]=\C \cdot [e_{n-2},e_{n-1}] =\C \cdot \alpha\, e_2.$$ \end{proof}

\begin{ejemplo}\label{example_z1=z2=n-2}
Let us assume that $z_1=z_2=n-2$. Then, one has $[C^2\fkg, C^4\fkg] =\{0\}$ due to the definition of $z_1$, see {\rm (\ref{defz1})}, and since $z_1^*=3$.
According to Proposition \ref{caso(n-2,n-2,n)new}, $[C^2\fkg, C^{z_1^{*}} \fkg] = [C^2\fkg,C^3\fkg]$ has dimension 1. In particular, $\theta_2=z_1^{*}+1=4$, see also Section \ref{theta-vector}. In Section \ref{hilbert_z2=n-2} we compute the Hilbert polynomial of these algebras.
\end{ejemplo}



\begin{crl}\label{ley(n-2,n-2,n)normalizada}
Let $\fkg$ be a filiform Lie algebras associated with the triple $(n-2,n-2,n)$. Then, there exists an adapted basis $\{ f_i\}_{i=1}^n$ such that the law of $\fkg$ is given by
$$[f_1,f_h]=f_{h-1}, \,\, {\rm for}\,\, 3 \leq h \leq n; \quad [f_{n-2},f_{n-1}]= f_2;\quad [f_{n-2},f_n]=f_3+\gamma\,'\,f_2;$$
$$[f_{n-1},f_n]=f_4+\gamma\,'\,f_3+\beta '\,f_2,$$ for some $\gamma\,', \beta' \in \C$.

\end{crl}

\begin{proof} 
We first apply Proposition \ref{caso(n-2,n-2,n)new}. Then, we apply Proposition \ref{cambiobase} with $\lambda=\alpha^{-1}$ and write $(\gamma\,',\beta')=\alpha^{-1}(\gamma,\beta)$.

\end{proof}

Moreover, we have also studied the different isomorphism classes for this family of filiform Lie algebras associated with the triple $(n-2,n-2,n)$.  

\begin{thr}\label{isomorfia(n-2,n-2,n)}
There exist only two non-isomorphic filiform Lie algebras associated with the triple $(n-2,n-2,n)$ for $n \geq 6$. These two algebras are given by the following non-zero brackets:
\begin{itemize}
	\item[1)] $[e_1,e_h]=e_{h-1}, \,\, h=3,\ldots,n; \,\, [e_{n-2},e_{n-1}]=e_2; \,\, [e_{n-2},e_n]=e_3+e_2; \,\, [e_{n-1},e_n]=e_4+e_3$. 
	\item[2)] $[e_1,e_h]=e_{h-1}, \, h=3,\ldots,n; \,\, [e_{n-2},e_{n-1}]=e_2; \,\, [e_{n-2},e_n]=e_3; \,\, [e_{n-1},e_n]=e_4.$
\end{itemize}
\end{thr}

\begin{proof} 
In order to simplify the notation, we will denote by $\alpha, \gamma,\beta$ the paramenters $\alpha_1,\gamma_1,\beta_{1,2}$, respectively.
Let $\fkg_{\alpha,\gamma,\beta}$ be a filiform Lie algebra associated with the triple $(n-2,n-2,n)$ and with adapted basis $\{ e_h\}_{h=1}^n$. The law of $\fkg_{\alpha,\gamma,\beta}$ is given by Proposition \ref{caso(n-2,n-2,n)new}. Notice that, as proved in \loccit, $\alpha \neq 0$. 
Moreover and according to Corollary \ref{ley(n-2,n-2,n)normalizada}, we can suppose that $\alpha=1$. Under those assumptions, we distinguish two cases. 

First, if $\gamma \neq 0$, then $\fkg_{1,\gamma,\beta}$ is isomorphic to the algebra $\fkg_{1,1,0}$ via the Lie algebra isomorphism
$\varphi:\fkg_{1,\gamma,\beta} \rightarrow  \fkg_{1,1,0} $ given by 
$$\varphi(e_1)=e'_1=\gamma e_1; \quad  \varphi(e_2)=e'_2=\gamma^{2n-7} \, e_2; \quad \varphi(e_3)=e'_3=\gamma^{2n-8} e_3;$$
$$\varphi(e_k)=e'_k=
\dfrac{\gamma^{2n-k-5} \, \beta}{2} \, e_{k-2}+\gamma^{2n-k-5} \, e_k, \quad {\rm for} \,\, 4\leq k \leq n$$
where $\{e'_h\}_{h=1}^n$ is the adapted basis of $\fkg_{1,1,0}$ given in Proposition \ref{caso(n-2,n-2,n)new}. Under these conditions, let us prove that we obtain the law given in $1)$.
$$[e'_2,e'_i]=0=[e'_3,e'_{j}]=[e'_k,e'_{\ell}], \quad \forall \, 1 \leq i \leq n, \,\, \forall \, 2 \leq j \leq n, \quad \forall \, 4 \leq k, \ell \leq n-2;$$
$$[e'_1,e'_3]=\gamma^{2n-7}\,e_2=e'_2, \quad [e'_1,e'_4]=\gamma^{2n-8}\,e_3=e'_3$$
$$[e'_1,e'_h]=\left[\gamma \, e_1, \dfrac{\gamma^{2n-h-5}\,\beta}{2}\,e_{h-2}+\gamma^{2n-h-5}\, e_h \right]=
\dfrac{\gamma^{2n-h-4}\,\beta}{2}\,e_{h-3}+\gamma^{2n-h-4}\, e_{h-1}$$
$$=e'_{h-1}, \quad {\rm for} \,\, 5 \leq h \leq n;$$
$$[e'_{n-2},e'_{n-1}]=\left[\dfrac{\gamma^{n-3}\,\beta}{2}\,e_{n-4}+\gamma^{n-3}\,e_{n-2},\dfrac{\gamma^{n-4}\,\beta}{2}\,e_{n-3}+\gamma^{n-4}\, e_{n-1}  \right]=\gamma^{2n-7}e_{2}=e'_2$$
$$[e'_{n-2},e'_{n}]=\left[\dfrac{\gamma^{n-3}\,\beta}{2}\,e_{n-4}+\gamma^{n-3}\,e_{n-2},\dfrac{\gamma^{n-5}\,\beta}{2}\,e_{n-2}+\gamma^{n-5}\, e_{n}  \right]=\gamma^{2n-8}(e_3+\gamma \, e_{2})=e'_3+e'_2$$
$$[e'_{n-1},e'_{n}]=\left[\dfrac{\gamma^{n-4}\,\beta}{2}\,e_{n-3}+\gamma^{n-4}\,e_{n-1},\dfrac{\gamma^{n-5}\,\beta}{2}\,e_{n-2}+\gamma^{n-5}\, e_{n}  \right]$$
$$=-\dfrac{\gamma^{2n-9} \beta}{2}\,e_2+\gamma^{2n-9}(e_4+\gamma\,e_3+\beta\,e_2)
=\gamma^{2n-9}e_4+\gamma^{2n-8}\,e_3+\dfrac{\gamma^{2n-9} \, \beta}{2}\,e_2=e'_4+e'_3.$$

In case that $\gamma = 0$, then $\fkg_{1,0,\beta}$ is isomorphic to the algebra $\fkg_{1,0,0}$ via the Lie algebra isomorphism
$\phi:\fkg_{1,0,\beta} \rightarrow  \fkg_{1,0,0} $ given by 
$$\phi(e_i)=e'_i=e_i, \,\, {\rm for}\,\, 1\leq i \leq 3; \quad  \phi(e_k)=e'_k=
\frac{\beta}{2} \, e_{k-2}+ e_k, \quad {\rm for} \,\, 4\leq k \leq n$$
where $\{e'_h\}_{h=1}^n$ is the basis of $\fkg_{1,0,0}$.  Under these conditions we obtain the law given in $2)$.
$$[e'_2,e'_i]=0=[e'_3,e'_{j}]=[e'_k,e'_{\ell}], \quad \forall \, 1 \leq i \leq n, \,\, \forall \, 2 \leq j \leq n, \quad \forall \, 4 \leq k, \ell \leq n-2;$$
$$[e'_1,e'_3]=e_2=e'_2, \quad [e'_1,e'_h]=\left[e_1, \frac{\beta}{2}\,e_{h-2}+ e_h \right]=
\frac{\beta}{2}\,e_{h-3}+ e_{h-1}=e'_{h-1}, \, \forall \, 4 \leq h \leq n;$$
$$[e'_{n-2},e'_{n-1}]=[\frac{\beta}{2}\,e_{n-4}+e_{n-2},\frac{\beta}{2}\,e_{n-3}+e_{n-1}]=[e_{n-2},e_{n-1}]=e_{2}=e'_2$$
$$[e'_{n-2},e'_{n}]=[\frac{\beta}{2}\,e_{n-4}+e_{n-2},\frac{\beta}{2}\,e_{n-2}+e_{n}]=[e_{n-2},e_{n}]=e_{3}=e'_3$$
$$[e'_{n-1},e'_{n}]=[\frac{\beta}{2}\,e_{n-3}+e_{n-1},\frac{\beta}{2}\,e_{n-2}+e_{n}]=-\frac{\beta}{2}e_2+e_4+\beta e_2=e_4+\frac{\beta}{2} e_2=e'_4.$$
Finally, a straightforward computation shows that no isomorphism exists between $\fkg_{1,1,0}$	and $\fkg_{1,0,0}$.

\end{proof}

\subsubsection{Filiform Lie algebras associated with $(z_1,n-2,n)$ where $z_1<n-2$} \label{(z1,n-2,n)z1<n-2}


Now, we continue analyzing the case where $z_1(\fkg)<z_2(\fkg)=n-2$. Since $z_1(\fkg)\geq 4$ we are assuming in this subsection $n=\dim(\fkg)\geq 7$. 
In the following Lemma we prove that a term in a family of Jacobi identities is always zero. 

\begin{lmm}\label{lemaJacobi}
Let $\fkg$ be a filiform Lie algebra associated with the triple $(z_1,n-2,n)$, 
with $z_1 < n-2$. 
Then, for $0 \leq i \leq n-2-z_1$ the Jacobi identity
$J(e_{z_1+i},e_{n-1},e_n)=0$  
is given by
$$[[e_{z_1+i},e_{n-1}],e_n]+[[e_{n},e_{z_1+i}],e_{n-1}]=0.$$
Moreover, both terms in the previous expression are null in case that $z_1 >4$ and $1 \leq i \leq z_1-4$. 
\end{lmm} 

\begin{proof}
According to Theorem \ref{leygeneral}, there are complex numbers $\alpha_i$, $\gamma_1$ and $\beta_{k,2}$ such that the law of $\fkg$ is given by
\begin{equation}\label{law(z_1,n-2,n)} 
	\begin{aligned}    [e_1,e_h]&= e_{h-1} \mbox{ \rm{ for }} 3 \leq h \leq n, \\
		[e_{z_1+i},e_{n-1}]&=\alpha_1 e_{i+2}+\alpha_2 e_{i+1}+\cdots+\alpha_{i+1} e_2
		\mbox{ \rm{ for }} 0 \leq i \leq n-2-z_1,  \\
		[e_{z_1},e_{n}]&=\alpha_1 e_{3}+\gamma_1\, e_2,  \\
		[e_{z_{1}+k},e_{n}]&= \sum_{h=2}^{k+2} P_h\left([e_{z_{1}+k-1},e_{n}] +[e_{z_{1}+k},e_{n-1}]\right) e_{h+1}+ \beta_{k, 2} \, e_2, \\ &\phantom{= }{\mbox{\rm{\ for }}} 1 \leq k < n-z_{1}.
	\end{aligned}
\end{equation}
Consequently, $[e_{n-1},e_n]=[e_{z_1+n-z_1-1},e_n] \in \langle e_2, \ldots, e_{n-z_1+2} \rangle$. 
Note that $n-z_1+2 \leq  n-2$ since $z_1 \geq 4$. Therefore, for $0 \leq i \leq n-2-z_1$, one has $[[e_{n-1},e_n],e_{z_1+i}]=0$ due to the fact that $[e_{z_1+i},e_{q}]=0$, for $2\leq  q \leq  n-2$. 


Next, we assume that $4<z_1<n-2$ 
and we will prove that $$[[e_{z_1+i},e_{n-1}],e_n]=[[e_n,e_{z_1+i}],e_{n-1}]=0$$ for $1 \leq i \leq z_1-4$. 
First, one has
$$[[e_{z_1+i},e_{n-1}],e_n]=\alpha_1[e_{i+2},e_n]+\cdots+
\alpha_{i+1}[e_2,e_n]=\alpha_1 [e_{i+2},e_n]+\cdots+\alpha_{3-z_1+i}[e_{z_1},e_n]$$
due to the definition of the invariant $z_1$, see Equality (\ref{z_1-z_2-equivalent-def}). Moreover, if $3-z_1+i<1$, that is, if $i \leq z_1-3$, then $[[e_{z_1+i},e_{n-1}],e_n]=0$. Consequently, $[[e_{z_1+i},e_{n-1}],e_n]=0$ if $1 \leq i \leq z_1-4$. This proves the first assertion. 
On the other hand,
$$[[e_n,e_{z_1+i}],e_{n-1}]=-\displaystyle \sum_{h=2}^{i+2} P_h([e_{z_1},e_n]+[e_{z_1+i},e_{n-1}])\,[e_{h+1},e_{n-1}] - \beta_{i,2} \, [e_2,e_{n-1}]=$$
$$-\displaystyle \sum_{h=z_1-1}^{i+2} P_h(\alpha_1 e_3+\gamma_1 e_2+\alpha_1 e_{i+2}+\cdots+
\alpha_{i+1}e_2 )\,[e_{h+1},e_{n-1}].$$
In case that $z_1-1 > i+2$, i.e. $i \leq z_1-4$, then the previous expression is null. 
\end{proof}

Now, we continue studying filiform Lie algebras associated with the triple $(z_1,n-2,n)$ with $z_1 < n-2$. 
The following result characterizes the law of these algebras by describing its restrictions relative to the general law; see Theorem \ref{leygeneral}. 
This characterization will be used in Section \ref{hilbert_z2=n-2} to determine the Hilbert polynomial corresponding to this family of filiform Lie algebras.


\begin{notation}\label{p}
With previous notations, we write $p=p(\fkg)=\lfloor \frac{n-z_1-1}{2} \rfloor$ where $\lfloor\,\, \rfloor$ denotes the floor function. 
\end{notation}
\begin{prp}\label{caso(z_1,n-2,n)}
Let $\fkg$ be a filiform Lie algebra associated with the triple $(z_1,n-2,n)$, where $z_1 < n-2$. Then the law of $\fkg$ is defined by Theorem \ref{leygeneral} with closed restrictions $\{\alpha_i=0\}$ 
for $1 \leq i \leq p$ and open restrictions $\gamma_1\neq 0$ and 
$(\alpha_{p+1},\ldots,\alpha_{n-z_1-1})\neq (0, \ldots, 0).$ 

\end{prp}

\begin{rmr}\label{aclaracionPropo3}
The open restrictions of Proposition \ref{caso(z_1,n-2,n)} are deduced from the definition of invariants $z_1(\fkg)$ and $z_2(\fkg)$, while the closed restrictions are equivalent to  $[C^{n-z_1-p+2}\fkg, C^{n-z_2}\fkg]=0$.
\end{rmr}


\begin{proof} (of Proposition \ref{caso(z_1,n-2,n)})
First, the law of $\fkg$ is given at the beginning of the proof of Lemma \ref{lemaJacobi}. Since $z_2=n-2$, then $C^3\fkg=\langle e_2, \ldots,e_{n-2} \rangle$ is an abelian ideal and $C^2 \fkg$ is not. 
Consequently, the Jacobi identities 
$J(e_{q},e_r,e_u)=0$ are trivially satisfied (that is, we do not obtain closed restrictions from those identities) in the case $q,r,u \in \{1, \ldots, n-2\}$. Now, we have to impose the Jacobi identities $J(e_{z_1+i},e_{n-1},e_n)=0$ for $0 \leq i \leq n-2-z_1$. According to Lemma \ref{lemaJacobi}, $$J(e_{z_1+i},e_{n-1},e_n)=[[e_{z_1+i},e_{n-1}],e_n]+[[e_n,e_{z_1+i}],e_{n-1}].$$ 
Next, we distinguish several cases depending on the value of $i$:
\begin{itemize}
	\item 
	When $i=0$, we obtain 
	$$[[e_{z_1},e_{n-1}],e_n]+[[e_n,e_{z_1}],e_{n-1}]=[\alpha_1 e_2,e_n]-[\alpha_1 e_3+\gamma_1 e_2,e_{n-1}]\equiv 0.$$ 
	\item	In case that $i=1$, it is satisfied that $[[e_{z_1+1},e_{n-1}],e_n]=[\alpha_1 e_3+\alpha_2 e_2,e_n]=0$ and 
	\begin{align*}
		[[e_n,e_{z_1+1}],e_{n-1}]&=-\displaystyle \sum_{h=2}^3 P_h([e_{z_1},e_n]+[e_{z_1+1},e_{n-1}])\,[e_{h+1},e_{n-1}] - \beta_{1,2}[e_2,e_{n-1}] \\ &=
		-\displaystyle \sum_{h=2}^3 P_h(\alpha_1 e_3+\gamma_1 e_2+\alpha_1 e_3+\alpha_2 e_2)\,[e_{h+1},e_{n-1}] \\ &=-2\alpha_1 [e_4,e_{n-1}]-(\gamma_1+\alpha_2)[e_3,e_{n-1}] \\ &= -2\alpha_1 [e_4,e_{n-1}].
	\end{align*} 
	Since $[e_{z_1},e_{n-1}]=\alpha_1\,e_2$,
	the Jacobi identity $J(e_{z_1+1},e_{n-1},e_n)$ 
	equals $-2\alpha_1^2 e_2$  if $z_1=4$ and it is identically zero otherwise. 
	In this way, we conclude that $\alpha_1=0$ if $z_1=4$ and there are no closed restrictions when $z_1 > 4$.
	\item If $i=2$, we have that $$[[e_{z_1+2},e_{n-1}],e_n]=[\alpha_1\,e_4+\alpha_2\, e_3+\alpha_3 e_2,e_n]=\alpha_1 [e_4,e_n].$$ This last expression is zero when $z_1 >4$ due to the definition of $z_1$. If $z_1=4$ we also obtain that the previous term is null according to the restriction $\alpha_1=0$ from the case $i=1$. For the last term in the Jacobi identity $J(e_{z_1+2}, e_{n-1}, e_n)=0$, we obtain
	\begin{align*}
		[[e_n,e_{z_1+2}],e_{n-1}] &= -\displaystyle \sum_{h=2}^4 P_h([e_{z_1+1},e_n]+[e_{z_1+2},e_{n-1}])\,[e_{h+1},e_{n-1}] - \beta_{2,2}[e_2,e_{n-1}] \\ &= 
		-\displaystyle \sum_{h=2}^4 P_h\left(3\alpha_1 e_4+(\gamma_1+2\alpha_2) e_3+(\alpha_3+\beta_{1,2}) e_2\right)\,[e_{h+1},e_{n-1}] \\ &= 
		-3\alpha_1 [e_5,e_{n-1}]-(\gamma_1+2\alpha_2)[e_4,e_{n-1}] - (\alpha_3+\beta_{1,2}) [e_3,e_{n-1}].
	\end{align*}		
	
	So, the Jacobi expression $J(e_{z_1+2},e_{n-1},e_n)$ 
	equals $-3\alpha_1^2 \,e_2$ when $z_1=5$ and it vanishes if $z_1 > 5$. Hence, we deduce that $\alpha_1 = 0$ in the case $z_1 = 5$, while no additional restrictions arise when $z_1 > 5$.

\end{itemize}

Let us note that the restrictions obtained from the Jacobi identities depends on the value of the invariant $z_1$. From now on, we assume that $z_1=4$. Then, we know that $\alpha_1=0$ due the case $i=1$ and there are no closed restrictions from the case $i=2$.


When $i=3$, the Jacobi identity $J(e_{z_1+3},e_{n-1},e_n)=0$ must be fulfilled. In this case,
\begin{equation} \label{Jacobi-z1+3-(1)}
[[e_{z_1+3},e_{n-1}],e_n]=[\alpha_2 e_4+\alpha_3 e_3+\alpha_4 e_2,e_n]=\alpha_2 \gamma_1 e_2 
\end{equation} and
\begin{equation}\label{Jacobi-z1+3-(2)}
\begin{aligned}
	[[e_n,e_{z_1+3}],e_{n-1}]& =-\displaystyle \sum_{h=2}^5 P_h([e_{z_1+2},e_n]+[e_{z_1+3},e_{n-1}])\,[e_{h+1},e_{n-1}] - \beta_{3,2}[e_2,e_{n-1}]\\  
	&= -\displaystyle \sum_{h=2}^5 P_h\left((\gamma_1+3\alpha_2)\,e_4+(2\alpha_3+\beta_{1,2})\,e_3+(\alpha_4+\beta_{2,2})\,e_2\right)\,[e_{h+1},e_{n-1}]\\ &=  -(\gamma_1+3\alpha_2)\, [e_5,e_{n-1}]-(2\alpha_3+\beta_{1,2})\,[e_4,e_{n-1}]=
	-(\gamma_1+3\alpha_2)\alpha_2 \, e_2.
\end{aligned}
\end{equation}
Then, from expressions (\ref{Jacobi-z1+3-(1)}) and (\ref{Jacobi-z1+3-(2)}), we have $$J(e_{z_1+3},e_{n-1},e_n)=\alpha_2 \gamma_1 \, e_2 - (\gamma_1+3\alpha_2)\alpha_2 \, e_2=-3\alpha_2^2\, e_2$$ and we conclude that $\alpha_2=0$.

After these first cases, we continue the proof by induction on $i$,  which also corresponds to an induction on $k$, since $i$ can be either $2k$ in the even case or $2k-1$ in the odd case. Therefore, it is also an induction on $k$. 
In fact, we have a double induction hypothesis: first, we assume that for $i=2k-1$, imposing $J(e_{z_1+j},e_{n-1},e_n)=0$, for $0 \leq j \leq 2k-1$ leads to
$\alpha_1=\alpha_2=\cdots=\alpha_k=0$. Second, we suppose that the Jacobi identities $J(e_{z_1+\ell},e_{n-1},e_n)=0$ are trivially satisfied for even values of $\ell$, where $0 \leq \ell < i$.

In order to complete the proof for case $z_1=4$, we first prove that $J(e_{z_1+2k},e_{n-1},e_n)=0$ is trivially satisfied and second that, when imposing $J(e_{z_1+2k+1},e_{n-1},e_n)=0$, we obtain condition $\alpha_{k+1}=0$.

According to the induction hypothesis, $\alpha_1=\alpha_2=\ldots=\alpha_k=0$. Then, the law of $\fkg$ is given by 
\begin{align*} [e_1,e_h]&= e_{h-1}, \mbox{ \rm{ for }} 3 \leq h \leq n, \\
[e_{z_1+i},e_{n-1}]&=\alpha_{k+1} e_{i+2-k}+\cdots+\alpha_{i+1} e_2
\mbox{ \rm{ for }} k \leq i \leq n-2-z_1,  \\
[e_{z_1},e_{n}]&=\gamma_1\, e_2,  \\
[e_{z_1+1},e_{n}]&=\gamma_1\, e_3+\beta_{1,2}\,e_2,  \\
[e_{z_1+2},e_{n}]&=\gamma_1\, e_4+\beta_{1,2}\,e_3+\beta_{2,2}\,e_2,  \\
&  \phantom{x}\vdots\\	
[e_{z_{1}+t-1},e_{n}]&=\gamma_1\, e_{t+1}+\beta_{1,2}\,e_t+\beta_{2,2}\,e_{t-1}+\cdots+\beta_{t-1,2}\,e_2, \mbox{ \rm{ for }} 2 \leq t \leq n-z_1.
\end{align*}
Now, we prove that $J(e_{z_1+2k},e_{n-1},e_n)=0$ is trivially satisfied. Let us note that, obviously, $2k < n-z_1$. In this case, by Lemma \ref{lemaJacobi}, the first term of $J(e_{z_1+2k},e_{n-1},e_n)$ is
\begin{equation}\label{jac-1a-parte}
\begin{split}
	[[e_{z_1+2k},e_{n-1}],e_n]& =\alpha_{k+1}[e_{k+2},e_n]+\cdots+\alpha_{2k+3-z_1}[e_{z_1},e_n]\\ &=\alpha_{2k+3-z_1}\gamma_1 \,e_2\,+ \alpha_{2k+2-z_1}(\gamma_1 \,e_3+\beta_{1,2}\,e_2) \\ &+\alpha_{2k+1-z_1}(\gamma_1 e_4+\beta_{1,2}\,e_3+\beta_{2,2}\,e_2)+ \cdots \\ & 
	+\alpha_{k+1}(\gamma_1 \,e_{k+4-z_1}+\beta_{1,2}\,e_{k+3-z_1}+\cdots+\beta_{k+2-z_1,2}\,e_2).
\end{split}
\end{equation}
Concerning the second term of that Jacobi identity, we have
$$[[e_n,e_{z_1+2k}],e_{n-1}]=-\displaystyle \sum_{h=z_1+k-1}^{2k+2} P_h([e_{z_1+2k-1},e_n]+[e_{z_1+2k},e_{n-1}])[e_{h+1},e_{n-1}],\,\,{\rm where}$$
$$[e_{z_1+2k},e_{n-1}]=\alpha_{k+1}e_{k+2}+\cdots+\alpha_{2k+1}e_2$$ and 
$$[e_{z_1+2k-1},e_n]=\gamma_1 \, e_{2k+1}+\beta_{1,2}\, e_{2k}+\cdots+\beta_{2k-1,2}\,e_2.$$
Let us note that  \, $P_h([e_{z_1+2k},e_{n-1}])=0$ \, since \, $z_1+k-1 \leq h \leq 2k+2$ \, and \, $k+2 < z_1+k-1$ due to the fact that $z_1 \geq 4$. Moreover, \, $P_{2k+2}([e_{z_1+2k-1},e_n])=0$ \, since \, $[e_{z_1+2k-1},e_n] \in \langle e_2,\ldots,e_{2k+1} \rangle$.
Then, 
\begin{equation}\label{jac-2a-parte}
\begin{aligned}
	[[e_n,e_{z_1+2k}],e_{n-1}]& =-\beta_{k+2-z_1,2}\,\alpha_{k+1} \, e_2 - \beta_{k-z_1+1,2} \,(\alpha_{k+1}\,e_3+\alpha_{k+2}\,e_2)  \\
	& -\beta_{k-z_1,2} \,(\alpha_{k+1}\,e_4+\alpha_{k+2}\,e_3+\alpha_{k+3}\,e_2) - \cdots  \\
	&-\beta_{1,2}\,(\alpha_{k+1}\,e_{k+2-z_1}+\cdots+\alpha_{2k+2-z_1}\,e_2)  \\ 
	&-\gamma_1(\alpha_{k+1}\,e_{k+4-z_1}+\cdots+ \alpha_{2k+3-z_1}\,e_2).
\end{aligned}
\end{equation}
Therefore, the Jacobi identity $J(e_{z_1+2k},e_{n-1},e_n)=0$ is trivially satisfied since adding (\ref{jac-1a-parte}) and (\ref{jac-2a-parte}) one obtains $0$. 
Next, we prove that, by imposing $J(e_{z_1+2k+1},e_{n-1},e_n)=0$, we obtain condition $\alpha_{k+1}=0$. The first term on this Jacobi identity is
$$[[e_{z_1+2k+1},e_{n-1}],e_n]=\alpha_{k+1}[e_{k+3},e_n]+\cdots+\alpha_{2k+4-z_1}[e_{z_1},e_n]=\alpha_{2k+4-z_1}\gamma_1 \,e_2+$$
$$\alpha_{2k+3-z_1}(\gamma_1 \,e_3+\beta_{1,2}\,e_2)+\alpha_{2k+2-z_1}(\gamma_1 e_4+\beta_{1,2}e_3+\beta_{2,2}e_2)+ \cdots $$
$$+\alpha_{k+1}(\gamma_1 e_{k+5-z_1}+\beta_{1,2}e_{k+4-z_1}+\cdots+\beta_{k+3-z_1,2}e_2).$$
Regarding the other non-zero term of this Jacobi identity, we have
$$[[e_n,e_{z_1+2k+1}],e_{n-1}]=-\displaystyle \sum_{h=z_1+k-1}^{2k+3} P_h([e_{z_1+2k},e_n]+[e_{z_1+2k+1},e_{n-1}])[e_{h+1},e_{n-1}]$$\,\,{\rm where}
$$[e_{z_1+2k+1},e_{n-1}]=\alpha_{k+1}e_{k+3}+\cdots+\alpha_{2k+2}e_2$$ \, and \,
$$[e_{z_1+2k},e_n]=\gamma_1 e_{2k+2}+\beta_{1,2}e_{2k+1}+\cdots+\beta_{2k,2}e_2.$$
Since $z_1=4$, $P_h([e_{z_1+2k+1},e_{n-1}])=\alpha_{k+1}$ if $h=z_1+k-1$ and
$$[[e_n,e_{z_1+2k+1}],e_{n-1}]=-\displaystyle \sum_{h=z_1+k-1}^{2k+3} P_h([e_{z_1+2k},e_n]) [e_{h+1},e_{n-1}]-\alpha_{k+1}[e_{z_1+k},e_{n-1}].$$
Following a similar reasoning to the one used for the Jacobi identity $J(e_{z_1+2k},e_{n-1},e_n)=0$, 
one proves that  
$$[[e_{z_1+2k+1},e_{n-1}],e_n]=\displaystyle \sum_{h=z_1+k-1}^{2k+3} P_h([e_{z_1+2k},e_n])[e_{h+1},e_{n-1}].$$
Therefore, the Jacobi identity $J(e_{z_1+2k+1},e_{n-1},e_n)=0$ is given by
$$-\alpha_{k+1}[e_{z_1+k},e_{n-1}]=-\alpha_{k+1}^2 e_2=0,$$
and we conclude that $\alpha_{k+1}=0$.

Finally, we assume that $z_1 >4$. In that case, and according to Lemma \ref{lemaJacobi}, we know that $J(e_{z_1+i},e_{n-1},e_n)=0$ is trivially satisfied for $1 \leq i \leq z_1-4$. Following an analogous procedure  to the one used in this proof for the case $z_1=4$, it is straightforward 
to prove that, when imposing $J(e_{2z_1-4+j},e_{n-1},e_n)=0$ for $1 \leq j \leq n-2z_1+2$ and $j$ odd, we obtain $\alpha_j=0$. 

\end{proof}


\begin{ejemplo}
Regarding the proof of Proposition \ref{caso(z_1,n-2,n)}, 
in the following diagram we show as an example the restrictions obtained when $4 \leq z_1 \leq 8$ depending on the value of $i$, for $0\leq i\leq n-z_1-2$,  where the symbol $-$ means that no closed restrictions are obtained from the corresponding Jacobi identity and $k = \lfloor \frac{n-z_1-2}{2} \rfloor$.  


{\rm $$\begin{tabular}{|c|c|c|c|c|c|}
\hline
$i \backslash z_1$ & 4 &5  & 6 & 7 & 8 \\
\hline
1& $\alpha_1=0$ & --  & --  & --  & --  \\
\hline
2& -- & $\alpha_1=0$  & -- & -- & -- \\
\hline
3& $\alpha_2=0$ & -- &  $\alpha_1=0$ & -- & --  \\
\hline
4& -- & $\alpha_2=0$ & -- &  $\alpha_1=0$ & -- \\
\hline
5& $\alpha_3=0$ & -- & $\alpha_2=0$ & -- &  $\alpha_1=0$ \\
\hline
6& --  & $\alpha_3=0$ & -- & $\alpha_2=0$ & -- \\
\hline
7& $\alpha_4=0$ & -- & $\alpha_3=0$ & -- & $\alpha_2=0$ \\
\hline
8& -- & $\alpha_4=0$ & -- & $\alpha_3=0$ & -- \\
\hline
\vdots&  \vdots & \vdots & \vdots &\vdots  & \vdots  \\
\hline
$2 k$& -- & $\alpha_k=0$ & -- & $\alpha_{k-1}=0$ & --  \\
\hline
$2 k+1$& $\alpha_{k+1}=0$ & -- & $\alpha_k=0$ & --  & $\alpha_{k-1}=0$ \\
\hline
\end{tabular}$$} \vspace{0.2cm}

\end{ejemplo}






Using Proposition \ref{cambiobase} and Remark \ref{remarkcambiobase}, it is possible to reduce the number of parameters in the law of a filiform Lie algebra fulfilling the hypothesis of Proposition \ref{caso(z_1,n-2,n)}.

\begin{crl}\label{ley(z_1,n-2-n)normalizada} 
Under the hypothesis of Proposition \ref{caso(z_1,n-2,n)},
there exists an adapted basis $\{ f_i\}_{i=1}^n$ such that the law of $\fkg$ is given by
\begin{align*} [f_1,f_h]&= f_{h-1} \mbox{ \rm{ for }} 3 \leq h \leq n, \\
[f_{z_1},f_{n}]&= f_2,  \\
[f_{z_1+i},f_{n-1}]&=\frac{\alpha_{p+1}}{\gamma_1} f_{i-p+2}+\frac{\alpha_{p+2}}{\gamma_1} f_{i-p+1}+\ldots+\frac{\alpha_{i+1}}{\gamma_1} f_2
\mbox{ \rm{ for }} p \leq i \leq n-2-z_1,  \\
[f_{z_{1}+k},f_{n}]&= \sum_{h=2}^{k+2} P_h\left([f_{z_{1}+k-1},f_{n}] +[f_{z_{1}+k},f_{n-1}]\right) \gamma_1 f_{h+1}+ \frac{\beta_{k,2}}{\gamma_1} \, f_2, \\ &\phantom{= }{\mbox{\rm{\ for }}} 1 \leq k \leq n-z_{1}-1.
\end{align*}
\end{crl}

\begin{proof}
Take $z_1<z_2=n-2$, $\alpha_1= 0,$ and $ q=1$ in Proposition \ref{cambiobase} and Remark \ref{remarkcambiobase}.	
\end{proof}

\begin{rmr} \label{nulidad_c2cz1*}
Notice that, by Theorem \ref{leygeneral}, if $z_2=n-2$ then  $$[C^2 \fkg, C^{z_1^{*}} \fkg]=[\langle e_2, \ldots, e_{n-1} \rangle, \langle e_2, \ldots, e_{z_1}  \rangle]=\C \cdot [e_{z_1},e_{n-1}] =\C \cdot \alpha_1 e_2 $$ for some $\alpha_1\in \C$. If $z_1=z_2=n-2$, then $[C^2 \fkg, C^{z_1^{*}} \fkg] \neq \{0\}$ due to Proposition \ref{caso(n-2,n-2,n)new}. In case that $z_1 < n-2$, due to Corollary \ref{ley(z_1,n-2-n)normalizada}, one has that  $[C^2 \fkg, C^{z_1^{*}} \fkg]=\{0\}$. See Section \ref{hilbert_z2=n-2}.



\end{rmr}




\begin{crl}\label{sobre_theta2}
Under the hypothesis of Proposition \ref{caso(z_1,n-2,n)}, one has   
$$  [C^2 \fkg, C^{z_1^{*}-p+1} \fkg] = \{0\}.$$
Then, $\theta_2(\fkg) \leq z_1^{*}-p+1$ and,   therefore, $[C^2 \fkg, C^{z_1^{*}-q} \fkg]=\{0\}$, for $0 \leq q \leq p-1.$ 

\end{crl}
\begin{proof} One has 
\begin{align*}
[C^2 \fkg, C^{z_1^{*}-p+1} \fkg]& =[\langle e_2, \ldots, e_{n-1} \rangle,\langle e_2, \ldots, e_{z_1+p-1} \rangle] \\ &= \langle [e_{z_1},e_{n-1}],\ldots ,[e_{z_1+p-1},e_{n-1}]\rangle \\ &=
\langle \alpha_1 \, e_2,\,  \alpha_1 \,e_3+\alpha_2\, e_2, \, \ldots, \alpha_1\, e_{p+1}\,+\,\alpha_2 \,e_p + \cdots+\alpha_p \,e_2  \rangle \\ &=\{0\}.
\end{align*}    
\end{proof}

Next, we show the following result concerning an isomorphism that allows us to reduce one parameter in the family of filiform Lie algebras associated with the triple $(n-3,n-2,n)$. See also Proposition \ref{isomorfia(n-r,n-2,n)}. 

\begin{prp}\label{isomorfia(n-3,n-2,n)} 

Let $\fkg$ be a  filiform Lie algebra associated with the triple $(n-3,n-2,n)$ for $n \geq 7$. Then this algebra is isomorphic to the filiform Lie algebra with non-zero brackets:
$$[e_1,e_h]=e_{h-1}, \,\, h=3,\ldots,n; \,\, [e_{n-3},e_{n}]=e_2; \,\, [e_{n-2},e_{n-1}]=\alpha\,e_2; \,\,$$ 
$$ [e_{n-2},e_n]=(1+\alpha)\,e_3+\beta\,e_2; \,\,[e_{n-1},e_n]=(1+\alpha)\,e_4+\beta\,e_3,$$
where $\alpha, \beta  \in \C$ and $0 \neq \alpha$.
\end{prp}

\begin{proof}
According to Theorem \ref{leygeneral}, the law of $\fkg$ is given by
$$[e_1,e_h]=e_{h-1}, \,\, \forall 3 \leq h \leq n; \quad
[e_{n-3},e_{n-1}]=\alpha_1\,e_2;\quad
[e_{n-2},e_{n-1}]=\alpha_1\,e_3+\alpha_2\,e_2; \,\,$$ 
$$[e_{n-3},e_{n}]=\alpha_1\, e_3+\gamma_1 \,e_2; \quad
[e_{n-2},e_n]=2\alpha_1\,e_4+(\gamma_1+\alpha_2)\,e_3+\beta_{1,2}\,e_2;$$ 
$$[e_{n-1},e_n]=2\alpha_1\,e_5+(\gamma_1+\alpha_2)\,e_4+\beta_{1,2}\,e_3+\beta_{2,2}\,e_2.$$

From Propositions \ref{cambiobase} and \ref{caso(z_1,n-2,n)} and Remark \ref{remarkcambiobase}, we can assume  $\alpha_1=0, \alpha_2 \neq 0$ and $\gamma_1=1$. Now, we denote by $\fkg_{\alpha_2,1,\beta_{1,2},\beta_{2,2}}$ that filiform Lie algebra.  
Then, this algebra is isomorphic to the filiform Lie algebra $\fkg_{\alpha_2,1,\beta_{1,2},0}$ 
via the Lie algebra isomorphism
$\varphi:\fkg_{\alpha_2,1,\beta_{1,2},\beta_{2,2}} \rightarrow  
\fkg_{\alpha_2,1,\beta_{1,2},0}$ given by 
$$\varphi(e_i)=e'_i=e_i, \quad \forall 1 \leq i \leq 3; \quad 
\varphi(e_j)=e'_j=\frac{\beta_{2,2}}{2\alpha_2} \, e_{j-2}+e_j, \quad \forall 4 \leq j \leq n$$
where $\{e'_h\}_{h=1}^n$ is an adapted basis of 
$\fkg_{\alpha_2,1,\beta_{1,2},0}$. 
First, it is obvious that
$$[e'_2,e'_i]=0=[e'_3,e'_{j}]=[e'_k,e'_{\ell}], \quad \forall \, 1 \leq i \leq n, \,\, \forall \, 2 \leq j \leq n, \quad \forall \, 4 \leq k, \ell \leq n-2;$$
Next,
$$[e'_1,e'_h]=[e_1,e_h]=e_{h-1}=e'_{h-1}, \quad 
[e'_{n-3},e'_{n}]=\left[\frac{\beta_{2,2}}{2\alpha_2}\,e_{n-5}+e_{n-3},
\frac{\beta_{2,2}}{2\alpha_2}\,e_{n-2}+e_{n}\right]=$$ 
$$[e_{n-3},e_n]=e_2=e'_2, \quad [e'_{n-2},e'_{n-1}]=\left[\frac{\beta_{2,2}}{2\alpha_2}\,e_{n-4}+e_{n-2},
\frac{\beta_{2,2}}{2\alpha_2}\,e_{n-3}+e_{n-1}\right]=[e_{n-2},e_{n-1}]=$$
$$\alpha_2\,e_2=\alpha_2\,e'_2; \quad 
[e'_{n-2},e'_{n}]=\left[\frac{\beta_{2,2}}{2\alpha_2}\,e_{n-4}+e_{n-2},
\frac{\beta_{2,2}}{2\alpha_2}\,e_{n-2}+e_{n}\right]=[e_{n-2},e_{n}]=$$
$$(1+\alpha_2)\,e_3+\beta_{1,2}\,e_2=(1+\alpha_2)\,e'_3+\beta_{1,2}\,e'_2;\,  
[e'_{n-1},e'_{n}]=\left[\frac{\beta_{2,2}}{2\alpha_2}\,e_{n-3}+e_{n-1},
\frac{\beta_{2,2}}{2\alpha_2}\,e_{n-2}+e_{n}\right]=$$
$$\frac{\beta_{2,2}}{2\alpha_2}\,[e_{n-3},e_n]-
\frac{\beta_{2,2}}{2\alpha_2}\,[e_{n-2},e_{n-1}]+[e_{n-1},e_n]=
\frac{\beta_{2,2}}{2\alpha_2}\,e_2-
\frac{\beta_{2,2}}{2\alpha_2}\,\alpha_2\,e_2+(1+\alpha_2)\,e_4+\beta_{1,2}\,e_3+\beta_{2,2}\,e_2$$
$$=(1+\alpha_2)(e_4+\frac{\beta_{2,2}}{2\alpha_2}\,e_2)+\beta_{1,2}\,e_3=
(1+\alpha_2)\,e'_4+\beta_{1,2}\,e'_3.$$

\end{proof}

In fact, Proposition \ref{isomorfia(n-3,n-2,n)} can be generalized as follows:

\begin{prp}\label{isomorfia(n-r,n-2,n)}
Let $\fkg$ be a filiform Lie algebra associated with the triple $(n-q,n-2,n)$ for $3 \leq q \leq n-4$ \,and $n\geq 7$. Then this algebra is isomorphic to the filiform Lie algebra with brackets:
$$[e_1,e_h]=e_{h-1}, \,\, \forall 3 \leq h \leq n; \,\,[e_{n-q+i},e_{n-1}]=\alpha_2\,e_{i+1}+\ldots+\alpha_{i+1}\,e_2, \,\,\forall \,\, 0\leq i \leq q-2;$$ 
$$ [e_{n-q},e_{n}]=e_2; \,\, 
[e_{n-q+k},e_n]= \sum_{h=2}^{k+2} P_h\left([e_{n-q+k-1},e_{n}] +[e_{n-q+k},e_{n-1}]\right) e_{h+1}+ \beta_{k,2} \, e_2, \,\, \forall \, 1 \leq k \leq q-2;$$
$$[e_{n-1},e_n]= \sum_{h=2}^{q+1} P_h\left([e_{n-2},e_{n}]\right) e_{h+1}, \quad {\rm where} \,\, (\alpha_2,\ldots,\alpha_{q-1}) \neq (0, \ldots, 0).$$
\end{prp}

\begin{proof}
The proof is similar to the one of Proposition \ref{isomorfia(n-3,n-2,n)}, using again Proposition \ref{caso(z_1,n-2,n)}, but considering the isomorphism 
$$\varphi:\fkg_{\alpha_2,\ldots,\alpha_{q-1},1,\beta_{1,2},\ldots,\beta_{q-1,2}} \longrightarrow  
\fkg_{\alpha_2,\ldots,\alpha_{q-1},1,\beta_{1,2},\ldots,\beta_{q-2,2},0}$$ given by 
$$\varphi(e_i)=e'_i=e_i, \quad \forall 1 \leq i \leq q; \quad 
\varphi(e_j)=e'_j=\frac{\beta_{q-1,2}}{(q-1)\alpha_2} \, e_{j-(q-1)}+e_j, \quad \forall q+1 \leq j \leq n.$$

\end{proof}

\subsection{Hilbert polynomial of a filiform Lie algebra with $z_2(\fkg)=n-2$}\label{hilbert_z2=n-2}

In this section, we compute the Hilbert polynomial associated with a $n$-dimensional filiform Lie algebra with $z_2(\fkg)=n-2$, that is, the algebras that were analyzed in Section \ref{Description}. First, we consider the case $z_1(\fkg)=n-2.$






\begin{prp}\label{Hp(n-2,n-2,n)}
Let $\fkg$ be an $n$-dimensional filiform Lie algebra associated with the triple $(n-2,n-2,n)$. Then, the Hilbert polynomial of $\fkg$ is given by
$$\HP_\fkg=\HP_\fkg^{(0)} + t^2\, s^3 + t^3\, s^2 + t^2\, s^2. $$

\end{prp}

\begin{proof} 
We have $z_1^*=z_2^*=3$. Then, by Remark \ref{z2*contraz2*-1-and-C2-C(z1*+1)}, one has $$[C^3\fkg,C^3\fkg]= [C^2\fkg, C^4\fkg]=\{0\}. $$ Moreover, by Proposition 
\ref{caso(n-2,n-2,n)new}, one has  
$$[C^2 \fkg, C^2 \fkg]=[\langle e_2, \ldots, e_{n-1} \rangle,
\langle e_2, \ldots, e_{n-1} \rangle]=\langle [e_{n-2},e_{n-1}] \rangle=\langle \alpha_1 e_2 \rangle \neq \{0\}.$$ In virtue of Lemma \ref{arrowshape}, $[C^2 \fkg, C^3 \fkg]=[C^3 \fkg, C^2 \fkg]=[C^2 \fkg, C^2 \fkg]$ and  $$\hp_{\fkg,2,2}= \hp_{\fkg,2,3} = \hp_{\fkg,3,2}=1.$$ 

\end{proof}

\begin{rmr} \label{hilbertpolynomials(n-2,n-2,n)}
Hilbert polynomials do not allow us to distinguish isomorphism classes within the family  of filiform Lie algebras associated with the triple $(n-2,n-2,n)$. However, see Theorem \ref{isomorfia(n-2,n-2,n)}.
\end{rmr}



\begin{prp}\label{dim-de-C2-C2} 
Let $\fkg$ be an $n$-dimensional filiform Lie algebra associated with the triple $(z_1,n-2,n)$, where $z_1 < n-2$. For $p=\lfloor \frac{n-z_1-1}{2} \rfloor$, we have 
\begin{align*} 
[C^2 \fkg, C^2 \fkg]&=
\langle \alpha_{p+1}\,e_2,\,
\alpha_{p+1}\,e_3+\alpha_{p+2}\,e_2, \ldots,\, 
\alpha_{p+1}\,e_{n-z_1-p} + \cdots + \alpha_{n-z_1-1}\,e_2 \rangle \neq \{0\}.\\
\end{align*}
And, therefore,
$$\dim([C^2\fkg,C^2\fkg])=\begin{cases} 
z_1^{*}-p-2 & \text{if } \, \alpha_{p+1} \neq 0 \\
z_1^{*}-p-3 & \text{if } \, \alpha_{p+1} =0,\, \alpha_{p+2} \neq 0 \\
\vdots & \vdots \\
1 & \text{if } \, \alpha_{p+1} = \alpha_{p+2} = \cdots = \alpha_{n-z_1-2}=0,\, \alpha_{n-z_1-1} \neq 0. 
\end{cases}
$$
\end{prp}


\begin{proof}
According to Theorem \ref{leygeneral} and Proposition \ref{caso(z_1,n-2,n)}, the law of $\fkg$ is given by
\begin{equation}\label{law(z_1,n-2,n)_propo3} 
\begin{aligned}  [e_1,e_h]&= e_{h-1}, \mbox{ \rm{ for }} 3 \leq h \leq n; \\
[e_{z_1+i},e_{n-1}]&=\alpha_{p+1} e_{i-p+2}+\cdots+\alpha_{i+1} e_2, \,\, \mbox{ \rm{ for }} p \leq i\leq n-2-z_1; \\
[e_{z_1},e_{n}]&=\gamma_1\, e_2; \\
[e_{z_{1}+k},e_n]&= \sum_{h=2}^{k+2} P_h\left([e_{z_1+k-1},e_n]+[e_{z_{1}+k},e_{n-1}]\right) e_{h+1}+ \beta_{k2} \, e_2, \, {\rm for}\, 1 \leq k < n-z_{1}; 
\end{aligned} 
\end{equation} 
where \begin{equation}\label{alpha-distinto-0}
(\alpha_{p+1}, \ldots, \alpha_{n-z_1-1}) \neq (0, \ldots, 0) 
\end{equation} 
and $\alpha_i,\, \gamma_1,\, \beta_{k,2} \in \C$. Then, the following equality holds
\begin{align*} 
[C^2 \fkg, C^2 \fkg]&=
[\langle e_2, \ldots, e_{n-1} \rangle,
\langle e_2, \ldots, e_{n-1} \rangle]=
\langle \{[e_{z_1+i},e_{n-1}]\,\vert \,  p \leq i \leq n-2-z_1  \} \rangle \\
&=  
\langle \alpha_{p+1}\,e_2,\,
\alpha_{p+1}\,e_3+\alpha_{p+2}\,e_2, \ldots,\, 
\alpha_{p+1}\,e_{n-z_1-p} + \cdots + \alpha_{n-z_1-1}\,e_2 \rangle.
\end{align*}
The previous bracket ideal is non-zero due to inequality (\ref{alpha-distinto-0}).

\end{proof}


\begin{prp}\label{saltos-de-altura-1}
Under the hypothesis of Proposition \ref{dim-de-C2-C2}, we have the following equality
$$\hp_{\fkg,2,\ell+1} = \dim [C^2 \fkg, C^{\ell+1}\fkg] = \dim [C^2 \fkg, C^{\ell}\fkg]-1= \hp_{\fkg,2,\ell}-1$$ for $3\leq \ell \leq \theta_2(\fkg)-1$. 
Moreover, $\theta_2(\fkg)= \dim [C^2\fkg,C^2\fkg]+3.$ 

\end{prp}

\begin{proof}  
The law of $\fkg$ is given by (\ref{law(z_1,n-2,n)_propo3}). Then, the following equalities hold:
\begin{align*}
[C^2 \fkg, C^{\ell} \fkg]&=[\langle e_2, \ldots, e_{n-1} \rangle,
\langle e_2, \ldots, e_{\ell^*} \rangle]\\
&=\langle \{[e_{z_1+i},e_{n-1}]\,\vert \,  p \leq i \leq z_1^{*}-\ell=\ell^*-z_1 \} \rangle \\
&=\langle \alpha_{p+1}\,e_2,\,
\alpha_{p+1}\,e_3+\alpha_{p+2}\,e_2, \ldots,\, 
\alpha_{p+1}\,e_{z_1^{*}-p-\ell+2} + \cdots + \alpha_{z_1^{*}-\ell+1}\,e_2 \rangle,
\end{align*}
for $3\leq \ell \leq  \theta_2(\fkg)-1$.


Let us write $m=z_1^*-\ell-p$. Let $M_\ell$ be the $m\times m$ matrix whose columns are the coordinates, in the basis $\{e_1,\ldots,e_n\}$,  of the vectors that generate $[C^2\fkg,C^\ell\fkg]$. Notice that the matrix $M_\ell$ is upper triangular and that its rank is precisely $\hp_{\fkg,2,\ell} = \dim \,[C^2\fkg, C^\ell \fkg].$
Moreover, the $k$-th diagonal of $M_{\ell}$ is precisely the vector $(\alpha_{p+k}, \ldots, \alpha_{p+k})$ with $m-k+1$ components, for $k=1,\ldots, m.$ 

The matrix $M_{\ell+1}$ is the principal submatrix of $M_\ell$ obtained by removing exactly the last row and column. Then $\rank\, M_{\ell} = \rank\, M_{\ell+1}+1$.

\end{proof}





Now, we show the main result of this section, where we compute the Hilbert polynomial of filiform Lie algebras under the condition stated in Section \ref{(z1,n-2,n)z1<n-2}, that is, $z_1(\fkg) < z_2(\fkg)=\dim \fkg -2 = n-2$.

\begin{thr}\label{Hp(z_1,n-2,n)}
Let $\fkg$ be an $n$-dimensional filiform Lie algebra associated with the triple $(z_1,n-2,n)$, where $z_1 < n-2$. Then, the Hilbert polynomial of $\fkg$ is given by
$$\HP_\fkg=\HP_\fkg^{(0)} +\dim\,[C^2\fkg,C^2\fkg]\,(t^2\, s^2+t^2 \,s^3+t^3 \,s^2) + \sum_{\ell=4}^{ z_1^{*}-p} \dim\,[C^2\fkg,C^\ell\fkg]\,(t^2 \,s^\ell+t^\ell \,s^2),$$ 
where $p=\lfloor \frac{n-z_1-1}{2} \rfloor$ and
$$\dim([C^2\fkg,C^{z_1^{*}-p-r}\fkg])=\begin{cases} 
r+1 & \text{if } \, \alpha_{p+1} \neq 0 \\
r & \text{if } \, \alpha_{p+1} =0, \alpha_{p+2} \neq 0 \\
\vdots & \vdots \\
1 & \text{if } \, \alpha_{p+1} =\alpha_{p+2}=\ldots =\alpha_{p+r}=0, \alpha_{p+r+1} \neq 0 \\
0 & \text{if } \, \alpha_{p+1} =\alpha_{p+2}=\ldots =\alpha_{p+r+1}=0
\end{cases}
$$
for $0 \leq r \leq z_1^{*}-p-4$.

\end{thr}


\begin{proof} According to Corollary \ref{sobre_theta2}, the definition of $\theta_2$ and Proposition \ref{dim-de-C2-C2}, $\dim \,[C^2 \fkg,C^{\theta_2-1} \fkg]=1$. Now, applying Proposition \ref{saltos-de-altura-1} and Lemma \ref{arrowshape}, we conclude that $\dim\,[C^2 \fkg, C^2 \fkg] = \dim\,[C^2 \fkg, C^3 \fkg]=\theta_2-3$. From Proposition  \ref{saltos-de-altura-1} we deduce that the difference between the dimension of two consecutive bracket ideals is one.

Moreover, the law of $\fkg$ is given by (\ref{law(z_1,n-2,n)_propo3}). Since $z_2^{*}=3$, then
$[C^k \fkg, C^{\ell} \fkg]=0$, for $k,\ell \geq 3$. According to Corollary \ref{sobre_theta2} and Remark \ref{nulidad_c2cz1*}, $[C^2 \fkg, C^{z_1^{*}-q} \fkg]=0$, for $0 \leq q \leq p-1$.
\begin{align*} 
[C^2 \fkg, C^{z_1^{*}-p} \fkg]&=[\langle e_2, \ldots, e_{n-1} \rangle,\langle e_2, \ldots, e_{z_1+p} \rangle]=[e_{z_1+p},e_{n-1}]=\langle \alpha_{p+1}e_2\rangle,\\
[C^2 \fkg, C^{z_1^{*}-p-1} \fkg]&=[\langle e_2, \ldots, e_{n-1} \rangle,\langle e_2, \ldots, e_{z_1+p+1} \rangle]=\langle [e_{z_1+p},e_{n-1}],[e_{z_1+p+1},e_{n-1}]  \rangle \\
&= \langle \alpha_{p+1}\,e_2,\alpha_{p+1}\,e_3+\alpha_{p+2}\,e_2 \rangle.
\end{align*}
More generally,
\begin{align*} 
[C^2 \fkg, C^{z_1^{*}-p-r} \fkg]&=[\langle e_2, \ldots, e_{n-1} \rangle,\langle e_2, \ldots, e_{z_1+p+r} \rangle]=\langle \{[e_{z_1+i},e_{n-1}]\,\vert \,  p \leq i \leq p+r  \} \rangle \\
&= \langle \alpha_{p+1}\,e_2,\alpha_{p+1}\,e_3+\alpha_{p+2}\,e_2,\ldots, \alpha_{p+1}e_{r+2} +\ldots + \alpha_{p+r+1}e_2 \rangle,\\
&{\rm for} \,\, 0 \leq r \leq z_1^{*}-p-4. 
\end{align*}
The previous expressions clarify the different cases written in the dimension of $[C^2 \fkg, C^{z_1^{*}-p-r} \fkg]$.





\end{proof}

\begin{crl}\label{theta2_r0}
Let $\fkg$ be a $n$-dimensional filiform Lie algebra associated with the triple $(z_1,n-2,n)$, where $z_1 < n-2$. Then, $\theta_2=z_1^{*}-p-r_0+1$, 
where $$r_0=\min\{r \,\, | \,\, [C^2 \fkg,C^{z_1^{*}-p-r} \fkg] \neq \{0\}, \,\, 0 \leq r \leq z_1^{*}-p-3 \}.$$ 
\end{crl} 

\begin{crl}\label{n-z1-p-1-classes}
Hilbert polynomials $\HP$ distinguishes $n-z_1-p-1\geq 1$ isomorphism classes of filiform Lie algebras associated with $(z_1,n-2,n)$ when $z_1 < n-2$.
\end{crl}

\section{Some sporadic examples of Hilbert polynomials}\label{sporadicHP}

In this section, we provide three examples of Hilbert polynomials that, in some cases,  allow us to distinguish between different isomorphism classes of filiform Lie algebras, or FLA,  associated with a fixed triple with $z_2(\fkg)=n-3$, for $n=8,9$, and $10$. Recall that we have denoted $\HP_\fkg^{(2)} = \HP_\fkg - \HP_\fkg^{(0)}$, see notation \ref{coeff-hp-and-k*}.

\subsection{$\HP$ of FLA associated with $(4,5,8)$} \label{(4,5,8)}

Let $\fkg$ be a filiform Lie algebra associated with the triple $(4,5,8)$. We have $z_1^*(\fkg)=5$, and $z_2^*(\fkg)=4$.
According to Theorem \ref{leygeneral}, the law of $\fkg$ is given, beside  by
\begin{align*} 
[e_1,e_h]&= e_{h-1} \mbox{ \rm{ for }} 3 \leq h \leq 8, \\
[e_{4+i},e_{6}]&=\alpha_1 e_{i+2}+\alpha_2 e_{i+1}+\cdots+\alpha_{i+1} e_2 \,\, \mbox{ \rm{ for }} 0 \leq i\leq 1, \\
[e_{4},e_{5+j}]&=\alpha_1 e_{j+1}+\gamma_1\, e_j+ \cdots+\gamma_{j-1}\,e_2 \,\,  \mbox{ \rm{ for }} 2 \leq j \leq 3,  \\
[e_{4+k},e_{5+\ell}]&= \sum_{h=2}^{k+\ell} P_h\left([e_{4+k-1},e_{5+\ell}] +[e_{4+k},e_{5+\ell-1}]\right) e_{h+1}+ \beta_{k \ell} \, e_2, \\ &\phantom{= }{\mbox{\rm{\ for }}} 2 \leq \ell \leq 3,\,\,  1 \leq k \leq \ell,
\end{align*} for certain vector of complex numbers $(\underline{\alpha},\underline{\gamma},\underline{\beta}) \in \C^2\times \C^2 \times \C^5=\C^9$.

After imposing the Jacobi identities $J(e_{q},e_r,e_u)=0$, for $4 \leq q < r < u \leq 8$, we obtain the following restrictions: $$\{\alpha_1=0,\alpha_2=-\gamma_1, \gamma_2=-\frac{5}{2}\beta_{12}, \gamma_1 \neq 0\}.$$ We denote $U$ the Zariski locally closed set in $\C^9$ defined by these restrictions and we identify each point of $U$ with its associated filiform Lie algebra. 
Since $z_2^{*}=4$, then $[C^4 \fkg,C^4 \fkg]=\{0\}$. Now, immediate computations show that 
$$ [C^3 \fkg,C^3 \fkg]=[C^3 \fkg,C^4 \fkg]=\langle e_2 \rangle, \quad    
[C^3 \fkg,C^5 \fkg] \subset C^8\fkg \,=\{0\}, 
$$ 
and so, $\theta_3(\fkg)=5.$ 

Moreover, 
\begin{align*}
[C^2 \fkg, C^2 \fkg] &= [C^2 \fkg, C^3 \fkg] = \langle e_2, \,\beta_{12} e_3\rangle, & 
[C^2 \fkg, C^4 \fkg]&=\langle e_2 \rangle, \\ [C^2 \fkg,C^5 \fkg]&=\langle e_2 \rangle, & [C^2 \fkg, C^6 \fkg]&=\{0\}.
\end{align*}
In particular, $\theta_2(\fkg)=6$. Moreover, $\hp_{\fkg,2,2}=\dim\,[C^2 \fkg,C^2 \fkg]$ depends on the vanishing of the parameter $\beta_{1,2}$.

Let us denote by $U'$ the Zariski open subset of $U$ defined by $\beta_{12}\neq 0$ and write $Z':=U\setminus U'.$ 

Then, $\hp_{\fkg,2,2}=2$ if $\fkg \in U'$ and $\hp_{\fkg,2,2} =1$ if $\fkg \in Z'$.   

So, if $\fkg\in U'$ then $$\HP_\fkg^{(2)}(t,s) = 2 (t^2s^2+t^2s^3+t^3s^2) + (t^2s^4+t^3s^3+t^4s^2) + (t^2s^5+t^3s^4+t^3s^4+t^5s^2),$$ while for $\fkg\in Z'$ one has    
$$\HP_\fkg^{(2)}(t,s) =  (t^2s^2+t^2s^3+t^3s^2) + (t^2s^4+t^3s^3+t^4s^2) + (t^2s^5+t^3s^4+t^3s^4+t^5s^2).$$

This means that the Hilbert polynomial identifies two different isomorphism classes of filiform Lie algebras associated with the triple $(4,5,8)$. These two classes cannot be distinguished by the $E$-invariant (see Section \ref{theta-vector}) or the numerical invariants $z_1$ and $z_2$.

\subsection{$\HP$ for FLA associated with $(5,6,9)$}\label{(5,6,9)}
Let $\fkg$ be a filiform Lie algebra associated with the triple $(5,6,9)$. We have $z_1^*(\fkg)=5$, and $z_2^*(\fkg)=4$.  According to Theorem \ref{leygeneral}, the law of $\fkg$ is given by
\begin{align*} [e_1,e_h]&= e_{h-1} \mbox{ \rm{ for }} 3 \leq h \leq 9, \\
	[e_{5+i},e_{7}]&=\alpha_1 e_{i+2}+\alpha_2 e_{i+1}+\cdots+\alpha_{i+1} e_2 \,\, \mbox{ \rm{ for }} 0 \leq i\leq 1, \\
	[e_{5},e_{6+j}]&=\alpha_1 e_{j+1}+\gamma_1\, e_j+ \cdots+\gamma_{j-1}\,e_2 \,\,  \mbox{ \rm{ for }} 2 \leq j \leq 3,  \\
	[e_{5+k},e_{6+\ell}]&= \sum_{h=2}^{k+\ell} P_h\left([e_{5+k-1},e_{6+\ell}] +[e_{5+k},e_{6+\ell-1}]\right) e_{h+1}+ \beta_{k \ell} \, e_2, \\ &\phantom{= }{\mbox{\rm{\ for }}} 2 \leq \ell \leq 3,\,\,  1 \leq k \leq \ell,
\end{align*}	for certain vector of complex numbers $(\underline{\alpha},\underline{\gamma},\underline{\beta}) \in \C^2\times \C^2 \times \C^5=\C^9.$
After imposing the Jacobi identities $J(e_{q},e_r,e_u)=0$, for $5 \leq q < r < u \leq 9$, we obtain the following family of restrictions:
$$U \equiv \{\alpha_1=0,\,(2\alpha_2+{3}\gamma_1)(\alpha_2-\gamma_1)=0, \,\gamma_1 \neq 0\}.$$ 
We notice that $U$ is a Zariski locally closed set in $\C^9$ and we identify each point of $U$ with its associated filiform Lie algebra. 
For $ \fkg \in U$, and since $z_2^*(\fkg)=4$, one has $[C^4\fkg, C^4\fkg]=\{0\}$. Immediate calculation shows 
$$[C^3 \fkg,C^3 \fkg]=[C^3 \fkg,C^4 \fkg]=\langle e_2 \rangle, \quad [C^3 \fkg,C^5 \fkg] =\{0\},$$ 
and then $\theta_3(\fkg)=5.$ 

We also have, 
$$[C^2 \fkg,C^6 \fkg]=\{0\},\quad [C^2 \fkg,C^5 \fkg]=\langle e_2\rangle, \quad [C^2 \fkg,C^4 \fkg]=\langle e_2,e_3 \rangle,$$  
and, in particular, $\theta_2(\fkg)=6$. 

Finally, we have 
$$[C^2 \fkg,C^2 \fkg] = [C^2 \fkg,C^3 \fkg] = \langle e_2, \,(\alpha_2+\gamma_1)e_3, \, \beta_{1,2}e_3+ (\alpha_2+\gamma_1)e_4\rangle =\langle e_2,e_3,e_4\rangle,$$ the last equality due to the fact that, for points in $U$, we have $\alpha_2+\gamma_1\neq 0$.


For each $\fkg\in U$ we have $$\HP_\fkg^{(2)} = 3(t^2s^2+t^2s^3+t^3s^2) + 2(t^2s^4+t^4s^2) + (t^2s^5+t^5s^2) + (t^3s^3+t^3s^4+t^4s^3).$$
The Hilbert polynomial does not distinguish non-isomorphic filiform Lie algebras associated with the triple $(5,6,9)$.

\subsection{$\HP$ for FLA associated with $(5,7,10)$}\label{(5,7,10)}

Let $\fkg$ be a filiform Lie algebra associated with the triple $(5,7,10)$. We have $z_1^*(\fkg)=6$, and $z_2^*(\fkg)=4$. 
According to Theorem \ref{leygeneral}, the law of $\fkg$ is given by
\begin{align*} [e_1,e_h]&= e_{h-1} \mbox{ \rm{ for }} 3 \leq h \leq 10, \\
	[e_{5+i},e_8]&=\alpha_1 e_{i+2}+\alpha_2 e_{i+1}+\cdots+\alpha_{i+1} e_2 \,\, \mbox{ \rm{ for }} 0 \leq i\leq 2, \\
	[e_{5},e_{7+j}]&=\alpha_1 e_{j+1}+\gamma_1\, e_j+ \cdots+\gamma_{j-1}\,e_2 \,\,  \mbox{ \rm{ for }} 2 \leq j \leq 3,  \\
	[e_{5+k},e_{7+\ell}]&= \sum_{h=2}^{k+\ell} P_h\left([e_{5+k-1},e_{7+\ell}] +[e_{5+k},e_{7+\ell-1}]\right) e_{h+1}+ \beta_{k \ell} \, e_2, \\ &\phantom{= }{\mbox{\rm{\ for }}} 2 \leq \ell \leq 3,\,\,  1 \leq k \leq  1+\ell,
\end{align*}	
for certain vector of complex numbers $(\underline{\alpha},\underline{\gamma},\underline{\beta}) \in \C^3\times \C^2 \times \C^7 = \C^{12}.$
After imposing the Jacobi identities $J(e_{q},e_r,e_u)$, for $5 \leq q < r < u \leq 10$, we obtain the following three families of restrict ions: 
$$U_1\equiv \{\alpha_1=\alpha_2=\gamma_1=0, \gamma_2 \neq 0, \alpha_3 \neq 0  \},$$
$$U_2\equiv \{ \alpha_1=\gamma_1=0,\alpha_3=-\frac{3}{7}\gamma_2+\frac{2}{7}\beta_{1,2}, \,\gamma_2\neq 0, \,(\alpha_2,\alpha_3) \neq (0,0) \},$$
$$U_3\equiv \{\alpha_1=0,\alpha_3=-\frac{3 \gamma_2}{5}-\beta_{1,2}, \gamma_1=-\alpha_2, \,(\alpha_2,\gamma_2)\neq(0,0), \,(\alpha_2,\alpha_3)\neq(0,0) \}.$$ Notice that each $U_i$ is a Zariski locally closed subset in $\C^{12}$. Write $U_i=Z_i\cap G_i$ where $Z_i$,  (resp. $G_i$) is the obvious Zariski closed set (resp. the obvious open set) in $\C^{12}.$ We can also write $U_i= Z_i \cap G$ where $G$ is the Zariski open set in $\C^{12}$ defined by the following two inequalities: $(\alpha_2,\alpha_3)  \neq (0,0)$, and $(\gamma_1,\gamma_2) \neq (0,0)$. In particular, $U:=U_1 \cup U_2 \cup U_3 = (Z_1 \cup Z_2 \cup Z_3)\cap G$ is also a locally closed set. 
Each point in $U$ is identified with its associated filiform Lie algebra. If $\Gamma \subset \{1,2,3\}$ is not the empty set, we define  $$U_\Gamma = \left(\bigcap_{i\in \Gamma} U_i\right)\setminus \left(\bigcup_{j\in \Gamma^c} U_j\right)$$ where $\Gamma^c$ is the complement of $\Gamma$ in $\{1,2,3\}.$ 

Notice that $U_{\Gamma}=\varnothing$ if $\Gamma$ has two elements. 

We now compute $\HP_\fkg^{(2)}(t,s)$ for $\fkg \in U_\Gamma$ for each $\Gamma$.   

Recall that for $\fkg \in U$ we have $z_2^*(\fkg)=4$. Then, by Lemma \ref{arrowshape}, one has $[C^4\fkg, C^4\fkg]=\{0\}$.  We also  have the following equalities $$[C^3\fkg, C^6\fkg] =\{0\},\, [C^3\fkg, C^5\fkg] = \langle \alpha_2 e_2\rangle, \,  [C^3\fkg, C^4\fkg] =[C^3\fkg, C^3\fkg] = \langle \alpha_2 e_2, \alpha_2e_3 + \alpha_3 e_2\rangle.$$ Moreover, we have  
\begin{align*}
	[C^2\fkg, C^7\fkg]&=\{0\}, \\ [C^2\fkg,C^6\fkg] &= \langle \gamma_1 e_2\rangle, \\ 
	[C^2\fkg,C^5\fkg] &=  \langle \gamma_1 e_2,\,  \alpha_2 e_2, \, (\alpha_2+\gamma_1)e_3 + \beta_{12}e_2\rangle, \\
	[C^2\fkg,C^4\fkg] &= [C^2\fkg,C^5\fkg] + \langle  \alpha_2e_3+\alpha_3e_2, \,(2\alpha_2+\gamma_1)e_4 + (\alpha_3+\beta_{12})e_3  + \beta_{22}e_2\rangle, \\
	[C^2\fkg,C^3\fkg] &= [C^2\fkg,C^2\fkg] = [C^2\fkg,C^4\fkg] + \langle (2\alpha_2+\gamma_1)e_5 + (\alpha_3+\beta_{12})e_4  + \beta_{22}e_3  + \beta_{32}e_2\rangle. 
\end{align*} 
In order to simplify the notation, we encode the coefficients $\hp_{\fkg,k,\ell}$ of the polynomial $\HP_\fkg^{(2)}$ into two vectors: $$\hp_{\fkg,3}:=(\hp_{\fkg,3,5}, \hp_{\fkg,3,4})\quad {\mbox{ and }} \quad \hp_{\fkg,2}:=(\hp_{\fkg,2,6}, \hp_{\fkg,2,5}, \hp_{\fkg,2,4}, \hp_{\fkg,2,3}),$$ having in mind that, by Lemma \ref{arrowshape}, one has $\hp_{\fkg,3,4}=\hp_{\fkg,3,3}$ and $\hp_{\fkg,2,3}=\hp_{\fkg,2,2}.$

For $\fkg \in U_{\{1,2,3\}}$ we have $\hp_{\fkg,3}=(0,1)$ and $\hp_{\fkg,2}=(0,1,2,3)$.

For $\fkg \in U_{\{1\}}$ we have $\hp_{\fkg,3}=(0,1)$ and the value of $\hp_{\fkg,2}$ is given by the following table:

\begin{align*}
	(0,0,2,3) & \quad {\mbox{ if }} \, \beta_{12}=0 \\
	(0,1,1,1) & \quad {\mbox{ if }} \, \beta_{12}\neq 0, \alpha_3+\beta_{12}=0, \beta_{2,2}=0  \\
	(0,1,1,2) & \quad {\mbox{ if }} \, \beta_{12}\neq 0, \alpha_3+\beta_{12}=0, \beta_{2,2} \neq 0  \\
	(0,1,2,3) & \quad {\mbox{ if }} \, \beta_{12}\neq 0, \alpha_3+\beta_{12} \neq 0.  \\
\end{align*}


For $\fkg \in U_{\{2\}}$ we have $\hp_{\fkg,3}=(1,2)$ and $\hp_{\fkg,2}=(0,2,3,4)$. 

For $\fkg \in U_{\{3\}}$ we have $\hp_{\fkg,3}=(1,2)$ and $\hp_{\fkg,2}=(1,1,3,4)$.

\begin{rmr}
	The Hilbert polynomial distinguishes 6 isomorphism classes in the family of filiform Lie algebras associated with the triple $(5,7,10)$.
	
\end{rmr}

\section*{Statements and Declarations}

Both authors have worked equally in every aspect: conceptualization, methodology, validation, formal analysis, investigation, preparation, funding acquisition, project administration and writing.

\noindent The authors declare that there are no data associated to this work.

\noindent The authors declare no conflicts of interests.

\section*{Acknowledgment}
This work has been partially supported by PID2020-117843GB-I00, PID2024-156912N, FEDER and FQM-326. The authors thank Mercedes Rosas for her useful comments and suggestions.


\end{document}